%% file: main.tex
\documentclass[hidelinks,onefignum,onetabnum,final]{siamart251216}

\input{paper_shared}

\ifpdf
\hypersetup{
  pdftitle={Unconditionally Long-Time Stable Variable-Step Second-Order ETD Schemes for the 2D Periodic Incompressible NSE},
  pdfauthor={H. Wang, X. Wang, and M. Zhang}
}
\fi

\begin{document}

\maketitle

\begin{abstract}
We develop an efficient, unconditionally stable, variable-step second-order exponential time-differencing (ETD) scheme for the 2D incompressible Navier–Stokes equations with periodic boundary conditions, together with an embedded adaptive time-stepping variant. The scheme is unconditionally uniform-in-time stable in the sense that the numerical vorticity admits a time-uniform bound in $L^\infty(0,\infty; L^2)$ whenever the external forcing term is uniformly bounded in time in $L^2$, for all Reynolds numbers and time step sizes.

In the periodic vorticity formulation, each step requires two diagonal heat-type solves in Fourier space and one scalar cubic solve. Beyond the standard ETD framework, the proposed scheme incorporates two recently developed ingredients: (i) a concurrent-correction auxiliary variable (ccSAV), which is essential for achieving second-order temporal accuracy, and (ii) a mean-reverting scalar auxiliary variable (mr-SAV) multistep formulation, which plays a central role in ensuring long-time stability.

The proposed methods overcome key limitations of existing approaches for the Navier–Stokes equations: classical Runge–Kutta schemes generally lack provable long-time stability, while IMEX and SAV-based BDF methods typically do not admit unconditional stability guarantees in the variable-step setting. Numerical experiments in two spatial dimensions confirm second-order temporal accuracy and uniform long-time stability, and demonstrate that the embedded adaptive scheme provides a practical local error indicator and automatic time-step selection at moderate Reynolds number. 
The numerical results also show that long integrations are essential in intermittent regimes: simulations over only $\mathcal{O}(10)$ eddy-turnover times may miss transitions between quiescent and bursting states. The present work therefore provides both a stability mechanism and a practical adaptive strategy for reliable long-time simulation of forced incompressible flows.
\end{abstract}


\section{Introduction}
The incompressible Navier--Stokes equations (NSE) constitute a fundamental mathematical model for viscous incompressible fluid flows and arise in a wide range of scientific and engineering applications, including geophysical flows, atmospheric and oceanic dynamics, turbulence, and climate science. \ignore{In this paper, we consider the NSE
\begin{subequations}\label{eqn:vp_nse}
\begin{align}
    &\frac{\partial \bm u}{\partial t} + \bm u \cdot \nabla \bm u = \nu \Delta \bm u - \nabla p + \bm f, \\
    &\nabla \cdot \bm u = 0,
\end{align}
\end{subequations}
posed on a rectangular domain $\Omega \subset \mathbb{R}^d$ ($d=2,3$) equipped with periodic boundary conditions. To ensure uniqueness and stability, we impose zero-mean conditions on the velocity field $\bm u=(u_1,\dots,u_d)$, the pressure $p$, and the external forcing $\bm f$. Here, $\nu>0$ denotes the kinematic viscosity, and $\bm f$ represents an external body force that is assumed to be uniformly bounded in time.
The periodic boundary condition enables us to eliminate the pressure by applying the classical Leray--Hopf projection, which can be calculated explicitly and efficiently in Fourier space.}{In this work, we focus on the two-dimensional case.} The 2D periodic NSE can equivalently be written in the vorticity--streamfunction formulation

\begin{equation}\label{eqn:nse_2d}
\begin{aligned}
    \omega_t + \nabla^\bot \psi \cdot \nabla \omega &= \nu \Delta \omega + f , \\
    -\Delta \psi &= \omega,  
\end{aligned}
\end{equation}
where $\omega$\ignore{$\omega=\nabla\times\bm u=\partial_yu_1-\partial_x u_2$} denotes the scalar vorticity, $\psi$ is the streamfunction,  and $f$ is the forcing term in this formulation and is related to the body force $\mathbf{f}$ in the velocity-pressure formulation via $f=(\nabla\times\bm f)\cdot\mathbf{k}=\partial_xf_2-\partial_yf_1$, where $\mathbf{k}$ is the unit vector perpendicular to the plane. The incompressibility constraint on the velocity $\mathbf{u}=\nabla^\perp\psi=(\partial_y\psi, -\partial_x\psi)$ is automatically satisfied, and the pressure variable is eliminated, making this formulation particularly attractive for numerical simulations in periodic domains.

A fundamental qualitative property of the NSE is the \emph{uniform-in-time boundedness} of solutions in the energy norm. In particular, when the external forcing is uniformly bounded in $L^2$, the velocity field admits a global-in-time $L^2$ bound, independent of whether the flow exhibits laminar, chaotic, or turbulent behavior \cite{frisch1995turbulence,constantin1988navier,foias2001navier,majda2006nonlinear,temam2012infinite}. 
In two dimensions, an analogous uniform-in-time bound holds for the enstrophy—the $L^2$ norm of the vorticity field \cite{constantin1988navier, temam2012infinite}.
These uniform estimates not only underpin the mathematical foundation of modern fluid dynamics—guaranteeing the existence of absorbing sets, global attractors, and invariant measures—but they also impose a strict physical mandate on numerical algorithms. 

Accurate numerical approximation of the NSE over long time intervals is essential in many applications, particularly when one is interested in the statistical \emph{climate} of the model, such as long-time statistical quantities, time-averaged observables, or asymptotic regimes, rather than short-time transient dynamics. For such purposes, it is highly desirable that numerical schemes inherit the uniform-in-time boundedness of the continuous system, that is, they are \emph{long-time structure preserving}. Numerical methods that fail to maintain this fundamental stability property may exhibit artificial energy growth, spurious numerical instabilities, or incorrect long-time statistics, even if they are formally consistent and stable over finite time intervals.

Beyond long-time stability, higher-order time integration schemes are also highly desirable, as they allow relatively large time steps while maintaining a prescribed error tolerance. Moreover, variable time-stepping strategies are essential for enhancing computational efficiency, enabling the use of larger time steps during slow temporal evolution and smaller steps when rapid dynamics are present.

The goal of this work is to develop and analyze numerical algorithms for the incompressible Navier--Stokes equations that enjoy the following three properties: (i) uniform-in-time bounds on the solution for all positive time under variable time stepping; (ii) second-order accuracy in time under variable step sizes; (iii) embedded adaptive time-stepping that combines long-time stability and automatic time-step control.

\medskip
\noindent
\textbf{Long-time stability and numerical discretizations.}
Early numerical investigations of long-time stability for the NSE focused primarily on fully implicit schemes. Under additional assumptions on the continuous solution, long-time stability results were established for classical methods such as the backward Euler and Crank--Nicolson schemes \cite{heywood1990finite,lubich1996runge,tone2006long,tone2007long}. To reduce computational cost while retaining stability, significant effort has been devoted to linearly implicit and semi-implicit methods. Representative results include unconditional stability and dissipativity analyses for semi-implicit Euler and projection-type schemes \cite{geveci1989convergence,simo1994unconditional,ju2002global}, as well as unconditional long-time stability results for higher-order schemes such as BDF2 \cite{heister2017unconditional}. Further developments can be found in \cite{hill2000approximation,rebholz2023long,contri2025error,garcia2025error}. Despite their favorable stability properties, these approaches typically require solving non-symmetric linear systems with variable coefficients at each time step, which can be computationally expensive.

To improve efficiency, it is natural to treat the nonlinear advection term explicitly while discretizing the viscous term implicitly. This strategy allows reuse of the same linear operator at each time step. For the two-dimensional NSE in vorticity form, long-time stability and statistical convergence results were established for first- and second-order IMEX schemes in \cite{gottlieb2012long,wang2012efficient}, with higher-order extensions in \cite{cheng2016long}. However, the stability of these IMEX schemes typically relies on time-step restrictions depending on the Reynolds number, which can severely limit their effectiveness. 

\medskip
\noindent
\textbf{Exponential time-differencing and SAV methodologies.}
Exponential time-differencing (ETD) methods provide an appealing alternative for stiff nonlinear systems by treating the linear part exactly and incorporating the nonlinear term via Duhamel's principle. ETD schemes offer high-order temporal accuracy and strong stability properties and can be efficiently implemented using FFT-based techniques in periodic domains. They have been applied to a wide class of semilinear parabolic equations \cite{du2004stability,du2005analysis,ju2018energy,du2021maximum,wang2025novel} as well as some nonlinear equations with mild nonlinearities \cite{chen2021energy,fu2022energy,fu2025higher,chen2025arxiv}. The last paper addresses the variable-step case. 
Recently, significant progress has been made in developing high-order ETD schemes that rigorously preserve energy dissipation laws and maximum bound principles (MBP) for various phase-field models \cite{du2019_nonlocal_ac_etd, ju2021_ifrk_mbp, ma2023_gl_etd, tang2016_ac_mbp, duan2024_bfs_etd}. Furthermore, exponential integrators have been successfully tailored for capturing the dynamics of nonlinear Schrödinger and wave equations \cite{bao2013_nls_ewi, bao2012_kgz_ewi, doi:10.1137/23M1615656}. Nevertheless, the application of ETD methods to the incompressible NSE is nontrivial. 
While several ETD-based approaches have been proposed \cite{rogallo1977illiac,zhang1987schema,kooij2018exponential,li2018fast,li2020adaptive}, rigorous results on their long-time stability remain unavailable.

In parallel, the scalar auxiliary variable (SAV) methodology, first developed by Shen, Xu, and Yang \cite{shen2018scalar} for gradient flows, and its variants have emerged as powerful tools for constructing energy-stable schemes for various dissipative systems \cite{shen2018scalar,shen2019new,akrivis2019energy,lin2019numerical,lin2020energy}. The versatility of the SAV approach is reflected in its numerous recent extensions, including the multiple SAV (MSAV) \cite{doi:10.1137/18M1166961}, generalized SAV \cite{ju2022_gsav_exp}, and pseudo-spectral SAV schemes \cite{ANTOINE2021110328}. These methods, often combined with high-order Runge--Kutta or Discontinuous Galerkin discretizations \cite{tang2022_rksav_dg, zheng2023_chemotaxis_sav}, provide robust stability for complex dissipative systems. However, most existing high-order energy-decreasing schemes \cite{ju2022_stabilized_esav, duan2024_bfs_etd} are primarily designed for gradient flows. Their application to the NSE under nontrivial external forcing—where the goal is to preserve uniform-in-time $L^\infty$ bounds—remains a significant open problem. A different extension that utilizes the energy conservation property of the nonlinear term in concurrent time-marching of the SAV and the velocity/vorticity field, developed by Yang \cite{yang2021new} under the name of ZEC, and by Li, Shen, and Liu \cite{LiShenLiu2022SAVPressure}, has achieved notable success for the NSE and related systems \cite{huang2023stability,huang2025stability,di2023variable,ji2024unified,zeng2023fully,zhang2024unified,zhang2025diffuse, han2025highly}. 
Recently, one of the authors and collaborators incorporated a linear mean-reverting term into this concurrent SAV-velocity/vorticity marching approach and developed a second-order in time scheme (mr-SAV) that guarantees uniform-in-time bounds under nontrivial forcing \cite{han2025highly,coleman2024efficient}. This scheme is highly efficient, requiring only two Stokes solves per time step. However, a long-time stable variable-step extension remains elusive. While variable-step strategies based on BDF or multi-step ETD formulations have been successfully analyzed with rigorous stability frameworks for phase-field models \cite{doi:10.1137/20M1384105, hou2023_pfc_bdf2_sav, hou2023_tfmb_sav} and general dissipative systems \cite{liao2021moc}, as well as the unforced NSE model \cite{di2023variable}, we are not aware of any higher-order variable-step schemes that preserve uniform-in-time energy bounds under general forcing in the literature so far.

\medskip
\noindent
\textbf{Contributions of this work.}
We develop efficient high-order time-stepping schemes for the incompressible NSE by integrating mean-reverting SAV \cite{coleman2024efficient,han2025highly} with the concurrent-correction SAV (ccSAV)  mechanism of Hou and Qiao \cite{hou2023implicit}, and classical exponential time-differencing multistep methods. The proposed schemes treat the nonlinear advection term explicitly through extrapolation from past fluid states, while the auxiliary variable is determined from a scalar algebraic equation coupled to the new vorticity update.

The main contributions are summarized as follows:
\begin{itemize}
    \item An unconditionally stable variable-step second-order ETD-mr-ccSAV scheme under periodic boundary conditions.
    \item Rigorous unconditional uniform-in-time $L^\infty(0,\infty;L^2)$ bounds under variable step sizes without time-step restriction.
    \item Explicit treatment of nonlinear advection while preserving long-time stability.
    \item An embedded adaptive scheme combining long-time stability with a practical local error indicator and automatic time-step selection.
\end{itemize}

To the best of our knowledge, this is the first second-order-in-time, variable-step scheme for the two-dimensional periodic NSE with general bounded forcing that treats the nonlinear advection explicitly and admits a rigorous uniform-in-time $L^2$
 stability bound without a time-step restriction.
 The emphasis of this paper is on the construction and unconditional long-time stability of the variable-step method. The numerical results verify the expected temporal order and demonstrate the practical relevance of the stability mechanism.

The remainder of the paper is organized as follows. Section~\ref{sec:2} introduces the reformulation and functional setting. Section~\ref{sec:3} presents the numerical algorithm. Section~\ref{sec:4} establishes unconditional long-time stability. Numerical experiments are reported in Section~\ref{sec:5}. Concluding remarks are given in Section~\ref{sec:6}.

\section{Preliminaries and mean-reverting ccSAV reformulation of the NSE}\label{sec:2}

In this section, we introduce the function spaces and an extended system for the incompressible Navier--Stokes equations, which form the foundation for our construction of efficient and unconditionally long-time stable numerical schemes.

\subsection{Function spaces}
We first recall the periodic Sobolev spaces on $\Omega=(0,2\pi)\times(0,2\pi)$ with average zero:
\[
\dot{H}_{per}^1(\Omega):=\{\phi\in H_{loc}^1(\mathbf{R}^2)\big|\int_{\Omega}\phi=0\textit{ and } \phi\textit{ is }2\pi-periodic\textit{ in each direction}\}.
\]
$\dot{H}_{per}^{-1}$ is defined as the dual space of $\dot{H}_{per}^1$, and we denote by $\langle \cdot , \cdot \rangle$ the duality pairing between $\dot{H}^1_{\mathrm{per}}$ and $H^{-1}_{\mathrm{per}}$ induced by the $L^2$ inner product, and by $\Vert \cdot \Vert$ the norm in $L^2(\Omega)$. We work in the mean-zero periodic Sobolev spaces
\ignore{\begin{equation}
    (\dot{H}_{\mathrm{per}}^{k}(\Omega))^d = \text{the closure of $d$-dimensional zero-mean trigonometric polynomials in } (H^k(\Omega))^d,
\end{equation}
and its divergence-free subspaces for $k=0,1$,
}
\begin{equation}
     H = \left\{\omega \in \dot{L}^2(\Omega) : \int_\Omega \omega\,dx = 0\right\},\quad V = \left\{\omega\in \dot{H}_{\mathrm{per}}^1(\Omega) : \int_\Omega \omega\,dx = 0 \right\}.
\end{equation}

\subsection{Extended system with mean-reverting concurrent-correction SAV}
To facilitate the development of efficient algorithms for the NSE \eqref{eqn:nse_2d} while achieving second-order accuracy and preserving long-time stability, we introduce an extended system of the NSE.
Denote $\mathcal L = -\Delta $, $\psi=(\mathcal{L})^{-1}\omega$, $B(\varphi, \omega)=\nabla^\bot \varphi \cdot \nabla \omega$. The NSE \eqref{eqn:nse_2d} can be rewritten as 
\[\frac{\partial \omega}{\partial t} + \nu \mathcal{L} \omega + B(\psi, \omega) = f.\]
We now introduce the extended system which incorporates a dynamic scalar auxiliary variable $r(t)$ and a mean-reverting term $\gamma r$:
\begin{subequations}\label{eqn:mr_sav_v2}
\begin{align}
    &\frac{\partial \omega}{\partial t} + \nu \mathcal{L} \omega + (1 - r^k)\,B(\psi, \omega) = f, \label{eqn:mr_sav_1_v2}\\
    &\frac{\mathrm{d} r}{\mathrm{d} t} + \gamma r = -(r-1)^{k-1}{\tilde{\gamma}}\langle B(\psi, \omega), \omega \rangle. \label{eqn:mr_sav_2_v2}
\end{align}
\end{subequations}
Here\footnote{We obtain a first-order scheme when $k=1$.}  $k=2$, $\gamma>0$ is a user-specified mean-reverting parameter, $\tilde{\gamma}>0$ is a user-specified scaling parameter, 
and the nonlinear term $B(\varphi, \omega)$ satisfies the conservation property
\[
 \langle B(\varphi, \omega), \omega\rangle = 0,
\]
$ \forall \omega\in V$ and $\varphi$ satisfying $(-\Delta) \varphi\in H$.

It is straightforward to verify that
\begin{equation*}
 |r(t)| \le |r_0| e^{-\gamma t} + \frac{\tilde{\gamma}}{\gamma}(1 - e^{-\gamma t})\|(1-r)^{k-1}\|_{L^\infty(0,t)}\|\langle B(\psi,\omega), \omega \rangle\|_{L^\infty(0,t)}.
\end{equation*}
When the nonlinear term is energy conservative, that is, $\langle B(\varphi, \omega), \omega \rangle \equiv 0$, we deduce that
\begin{equation*}
r(t) = r_0 e^{-\gamma t} \ \longrightarrow\  0,\quad \text{as } t\to \infty.
\end{equation*}
Furthermore, if $r(0)=0$, then $r(t)\equiv 0$ for all $t \ge 0$.\ignore{, and we recover the original model \eqref{eqn:p_ns}.
}
Even if the initial data do not vanish, or the trilinear term is not identically zero, as may occur in numerical approximations, the mean-reverting term $\gamma r$ will drive $r(t)$ back to its desirable value $0$. This motivates the terminology \emph{mean-reverting scalar auxiliary variable} (mr-SAV).
The auxiliary variable $r$ is evolved concurrently with the vorticity $\omega$; this motivates the term ccSAV.

Moreover, if $r$ contains an error of size $\tau$, the induced error in the $\omega$ equation is of second order when $k=2$. This property is particularly useful when $r$ is approximated with only first-order accuracy. The ccSAV mechanism was utilized by Hou and Qiao for gradient flows \cite{hou2023implicit}. Related higher-order sequential correction ideas also appeared in \cite{huang2022new}.

The introduction of $r(t)$ facilitates the explicit treatment of the nonlinear term, while the damping term enhances stability. For $\gamma>0$, {$\tilde{\gamma}=1$}, and $k=1$, the long-time stability of discretizations of \eqref{eqn:mr_sav_v2} via BDF2 and Gear's extrapolation was established in \cite{han2025highly,coleman2024efficient}. 
If $\gamma=0, {\tilde{\gamma}=1}$, and $k=1$, the system reduces to the standard SAV-ZEC formulation. This approach has been extensively studied and applied to various fluid models \cite{yang2021new,yang2021novel,zhang2024unified,li2022new}. However, in this case, the error introduced by explicit treatment of nonlinear advection may accumulate over long-time simulations, and $r(t)$ may deviate significantly from $0$, even when $r(0)=0$. Consequently, simulation results may become unreliable at large times. The introduction of $\gamma$ renders \eqref{eqn:mr_sav_2_v2} dissipative and alleviates this difficulty.

\section{Algorithm}\label{sec:3}
Building upon the mr-SAV extended system presented in \eqref{eqn:mr_sav_v2}, we proceed to develop a second-order exponential time-differencing (ETD) time-stepping scheme. We also introduce an embedded adaptive time-stepping version of our scheme by leveraging an embedded first-order scheme.

Given an arbitrary terminal time $T$ and a set of non-overlapping time nodes
$0 = t_0 < t_1 < \cdots < t_N = T$ with the $k$th time step size
$\tau_k = t_k - t_{k-1}$, we allow the time partition to be non-uniform.
Let $\Psi^{n}$ denote the numerical approximation of $\Psi(t)$ at time $t_n$,
and abbreviate $\Psi(t_{n} + \frac{1}{2}\tau_{n+1})$ as $\Psi^{n + \frac{1}{2}}$.
We further introduce the extrapolation formula for $\Psi^{n + \frac{1}{2}}$, given by
\begin{equation}\label{extrapolation}
  \tilde{\Psi}^{n + \frac{1}{2}} =
  \frac{\tau_{n+1} + 2\tau_{n}}{2\tau_{n}}\Psi^{n}
  - \frac{\tau_{n+1}}{2\tau_{n}} \Psi^{n-1}.
\end{equation}

\subsection{ETD-mr-ccSAV-MS2}

To obtain a second-order accurate approximation $\omega^{n+1}$, we approximate the nonlinear terms in \eqref{eqn:mr_sav_v2} using a second-order midpoint extrapolation and set $k=2$. The resulting scheme, {in differential form},
takes the form
\begin{equation}\label{eqn:2nd_msSAV_approx}
    \left\{
    \begin{aligned}
    &\frac{\partial \omega(t)}{\partial t} + \nu \mathcal{L}\omega(t)
    + (1 - (r^{n+1})^2) {\tilde{B}^{n+\frac{1}{2}}}
    = f^{n+\frac{1}{2}},\\
    &\frac{d r(t)}{dt} + \gamma r(t)
    = (1 - r^{n+1}) \big(\varphi_1(\tau_{n+1}\gamma)\big)^{-1}{{\tilde{\gamma}}}
    \Big\langle \varphi_1(\tau_{n+1} \nu \mathcal{L}) {\tilde{B}^{n+\frac{1}{2}}}, \omega^{n+1}\Big\rangle,
    \end{aligned}
    \right.
\end{equation}
for $t\in (t_{n},t_{n+1}]$, {where $\tilde{B}^{n+\frac{1}{2}}$ is the extrapolation of $B({\psi}^n, {\omega}^n)$ and $B({\psi}^{n-1}, {\omega}^{n-1})$ via \eqref{extrapolation}, given by }
{\begin{equation}\label{extrapolation_B}
  \tilde{B}^{n + \frac{1}{2}} =
  \frac{\tau_{n+1} + 2\tau_{n}}{2\tau_{n}}B({\psi}^{n}, {\omega}^{n})
  - \frac{\tau_{n+1}}{2\tau_{n}} B({\psi}^{n-1}, {\omega}^{n-1}).
\end{equation}}
Subsequently, by applying the variation-of-constants formula to \eqref{eqn:2nd_msSAV_approx} { we arrive at the following {\bf second-order ETD mean-reverting ccSAV MS2 scheme (ETD-mr-ccSAV-MS2o)}}
\begin{subequations}\label{eqn:mrGSAV_ETDMS2}
    \begin{align}
    &\omega^{n+1} = \varphi_0(\tau_{n+1} \nu \mathcal{L})\omega^{n}
    - \tau_{n+1}\big(1 - (r^{n+1})^2\big)\varphi_1(\tau_{n+1} \nu \mathcal{L})
    {\tilde{B}^{n+\frac{1}{2}}} \label{eqn:mrGSAV_ETDMS2_u}\\
    &\qquad\qquad\qquad\qquad\qquad
    + \tau_{n+1} \varphi_1(\tau_{n+1} \nu \mathcal{L})f^{n+\frac{1}{2}},\notag \\
    &r^{n+1} = \varphi_0(\tau_{n+1} \gamma) r^{n}
    + \tau_{n+1}(1 - r^{n+1}){{\tilde{\gamma}}}
    \Big\langle \varphi_1(\tau_{n+1} \nu \mathcal{L}){\tilde{B}^{n+\frac{1}{2}}} , \omega^{n+1} \Big\rangle,\label{eqn:mrGSAV_ETDMS2_q}
    \end{align}
\end{subequations}
where {the classical exponential functions associated with ETD algorithms}, $\varphi_0(\cdot)$ and $\varphi_1(\cdot)$, are defined by
\begin{equation}\label{eqn:exp_relat}
    \varphi_0(z) = \mathrm{e}^{-z},\quad \varphi_1(z) = \frac{1 - \mathrm{e}^{-z}}{z}.
\end{equation}

The above coupled scheme can be solved efficiently by substituting \eqref{eqn:mrGSAV_ETDMS2_u} into \eqref{eqn:mrGSAV_ETDMS2_q} and solving the resulting scalar cubic algebraic equation for $r^{n+1}$ first. The procedure is as follows:
\begin{enumerate}
    \item Calculate the intermediate variables (via two heat solves):
    \begin{equation}\label{eqn:step_1}
        \begin{aligned}
        & \omega^{n+1}_1 = \varphi_0(\tau_{n+1} \nu \mathcal{L})\omega^{n}
        + \tau_{n+1} \varphi_1(\tau_{n+1} \nu \mathcal{L})f^{n+\frac{1}{2}},\\
        & \omega^{n+1}_2 = \tau_{n+1}\varphi_1(\tau_{n+1} \nu \mathcal{L}){\tilde{B}^{n+\frac{1}{2}}}.
        \end{aligned}
    \end{equation}
    \item Compute the coefficients:
    \begin{equation}\label{eqn:step_2}
        A^{n+1} = \langle \omega_1^{n+1}, \omega_2^{n+1}\rangle,\quad
        B^{n+1} = \Vert \omega_2^{n+1} \Vert^2,\quad
        C^{n+1} = \varphi_0(\tau_{n+1} \gamma)\, r^n.
    \end{equation}
    \item Determine the auxiliary variable $r^{n+1}$ by finding the real root of the cubic polynomial with the smallest magnitude\footnote{Choose the positive root in case there are two roots with the same magnitude.}
    \begin{equation}\label{eqn:cubic}
       g(r)={{\tilde{\gamma}}}B^{n+1} r^3 - {{\tilde{\gamma}}}B^{n+1} r^2 + (1 + {{\tilde{\gamma}}}A^{n+1} - {{\tilde{\gamma}}}B^{n+1}) r -({{\tilde{\gamma}}}A^{n+1} - {{\tilde{\gamma}}}B^{n+1} + C^{n+1}).
    \end{equation}
    \item Obtain $\omega^{n+1}$ from
    \begin{equation}\label{eqn:step_4}
         \omega^{n+1} = \omega_1^{n+1} - \big(1-(r^{n+1})^2\big) \omega_2^{n+1}.
    \end{equation}
\end{enumerate}
It is evident that the ETD-mr-ccSAV-MS2o scheme is computationally efficient, since only two heat solves (with the associated FFT) and one scalar cubic equation is solved per time step.

\subsection{Solvability and root selection}
We briefly discuss the solvability of the cubic algebraic equation \eqref{eqn:cubic}. When the leading coefficient vanishes, i.e., $B^{n+1}=0$, we have $\omega_2^{n+1} = 0$, which further implies $A^{n+1}=0$. In this case, \eqref{eqn:cubic} reduces to a linear polynomial with leading coefficient equal to $1$, and solvability is guaranteed. When $B^{n+1} > 0$, \eqref{eqn:cubic} is a non-degenerate cubic polynomial and hence admits at least one real root.

We next discuss unique solvability of the auxiliary variable in the small-step regime.
Using the facts that $\varphi_0(z) \approx 1$ and $\varphi_1(z) \approx 1$ when $z \ll 1$, 
and the assumption that the ratio of adjacent time steps satisfies $\tau_{n+1}/\tau_n \approx 1$ for practical applications,
we obtain $B^{n+1} = \mathcal{O}(\tau_{n+1}^{2})$, $A^{n+1} = \mathcal{O}(\tau_{n+1})$, and $C^{n+1} = \mathcal{O}(\tau_{n+1})$.
Moreover, since
$$
g'(r) = 3{{\tilde{\gamma}}}B^{n+1}r^2 - 2{{\tilde{\gamma}}}B^{n+1}r + (1 + {{\tilde{\gamma}}}A^{n+1} - {{\tilde{\gamma}}}B^{n+1}),
$$
whose minimum value is attained at $r = \frac{1}{3}$, we deduce
$g'(r) \ge g'(\frac13)= 1 +{{\tilde{\gamma}}} A^{n+1} - \frac{4}{3}{{\tilde{\gamma}}} B^{n+1} > 0$ for sufficiently small time steps. This implies that $g(r)$ is non-decreasing on $\mathbb{R}$ and hence admits a unique root.

In practical computations, Newton's method can be used to solve \eqref{eqn:cubic}. Specifically, when $\tau_{n+1}$ is sufficiently small, $g'(r) > 0$ holds. In addition,
$g''(r) = 6{{\tilde{\gamma}}}B^{n+1}r - 2{{\tilde{\gamma}}}B^{n+1} = \mathcal{O}(\tau_{n+1}^{2})$.
Therefore, Newton's method enjoys quadratic convergence with a favorable coefficient.

When the cubic equation admits multiple real roots, we select the real root with the smallest magnitude, choosing the positive one in case of a tie. This selection is consistent with the small-step regime, where the cubic has a unique root near zero. In all numerical experiments, Newton’s method initialized at the previous auxiliary value $r^n$
 converges to this same branch. The unconditional stability estimate below only uses the algebraic relation satisfied by the selected root and does not require uniqueness of the cubic root.

\subsection{Formal second-order accuracy}

We now give a heuristic argument for the second-order accuracy of $\omega$. The numerical approximation of $r^{n+1}$ in \eqref{eqn:mrGSAV_ETDMS2_q} is formally first order, since $\varphi_1(\tau_{n+1}\nu \mathcal{L}) \approx I$ when $\tau_{n+1}$ is small, leading to the inner product term
$$
\Big\langle \varphi_1(\tau_{n+1} \nu \mathcal{L}){\tilde{B}^{n+\frac{1}{2}}}, \omega^{n+1} \Big\rangle
= \mathcal{O}(\tau_{n+1}),
$$
due to the skew symmetry of the nonlinear advection term.
Nevertheless, \eqref{eqn:mrGSAV_ETDMS2_u} formally achieves second-order accuracy due to the following two observations:
\begin{enumerate}
    \item $(r^{n+1})^2 \tau_{n+1}\varphi_1(\tau_{n+1} \nu \mathcal{L}) {\tilde{B}^{n+\frac{1}{2}}} = \mathcal{O}(\tau_{n+1}^3)$.
    \item The explicit Adams--Bashforth ETD multistep scheme is second order; see \cite{zhang1987schema}.
\end{enumerate}
Indeed, note that
$\tau_{n+1} \varphi_1(\tau_{n+1} \nu \mathcal{L}) {\tilde{B}^{n+\frac{1}{2}}}
= \int_0^{\tau_{n+1}} \varphi_0((\tau_{n+1}-s)\nu\mathcal{L})\,{\tilde{B}^{n+\frac{1}{2}}}\,ds$,
and ${\tilde{B}^{n+\frac{1}{2}}}$ is formally a second-order approximation of
$B({\psi}(t_n+\frac{\tau_{n+1}}{2}), {\omega}(t_n+\frac{\tau_{n+1}}{2}))$.
Furthermore, for two smooth functions $g,h$ together with a second-order approximation of $h$ at the midpoint
$h^{1/2}=h(\tau/2)+\mathcal{O}(\tau^2)$, we have
\begin{eqnarray*}
\int_0^\tau g(s)(h(s)-h^{1/2})
&=& \int_0^\tau g(s)(h(s)-h(\tau/2)) + \int_0^\tau g(s)(h(\tau/2)-h^{1/2})
\\
&=& \int_0^\tau (g(\tau/2)+\mathcal{O}(\tau))\big(h'(\tau/2)(s-\tau/2)+\mathcal{O}(\tau^2)\big)
+ \int_0^\tau g(s)\mathcal{O}(\tau^2)
\\
&=& \mathcal{O}(\tau^3).
\end{eqnarray*}
Hence, the local truncation error is formally of order $\mathcal{O}(\tau^3)$, implying second-order accuracy in $\omega$ despite the first-order accuracy in $r$.

A rigorous convergence proof for the variable-step method is beyond the scope of the present work.

\subsection{Embedded adaptive pair}

We now present an embedded adaptive time-stepping algorithm based on the long-time stable variable-step second-order scheme \eqref{eqn:mrGSAV_ETDMS2}. This embedded adaptive scheme enables a practical local error indicator and automatic time-step selection with negligible computational overhead while maintaining long-time stability.
See \cite{Fehlberg1969NASA, DormandPrince1980JCAM} for some historical notes on the development of embedded pairs. In the context of incompressible flows and the Navier-Stokes equations, various adaptive time-stepping strategies have been extensively investigated in the literature \cite{doi:10.1137/070688018, doi:10.1137/080728032, John10}, alongside a recent interesting 1--2 pair developed in \cite{DeCariaSchneier2021CMAME} for the 2D Navier-Stokes equations (without the $L^\infty(0,\infty;L^2)$ bound).

First, we construct a first-order scheme by considering a  discretization of  \eqref{eqn:mr_sav_1_v2}--\eqref{eqn:mr_sav_2_v2} with $k=1$, which yields
\begin{subequations}\label{eqn:mrGSAV_ETDMS1}
    \begin{align}
    &\bar{\omega}^{n+1} = \varphi_0(\tau_{n+1} \nu \mathcal{L})\omega^{n}
    - \tau_{n+1} (1 - \bar{r}^{n+1}) \varphi_1(\tau_{n+1} \nu \mathcal{L})
    {\tilde{B}^{n+\frac{1}{2}}} \label{eqn:mrGSAV_ETDMS1_u}\\
    &\qquad\qquad\qquad\qquad\qquad
    + \tau_{n+1} \varphi_1(\tau_{n+1} \nu \mathcal{L})f^{n+\frac{1}{2}},\notag \\
    &\bar{r}^{n+1} = \varphi_0(\tau_{n+1} \gamma) r^{n}
    - \tau_{n+1} {{\tilde{\gamma}}}\Big\langle \varphi_1(\tau_{n+1} \nu \mathcal{L})
    {\tilde{B}^{n+\frac{1}{2}}}, \bar{\omega}^{n+1} \Big\rangle,\label{eqn:mrGSAV_ETDMS1_q}
    \end{align}
\end{subequations}
It is straightforward to verify that $\bar{r}^{n+1}$ in \eqref{eqn:mrGSAV_ETDMS1_q} is also of order $\tau_{n+1}$. Consequently, $1 - \bar{r}^{n+1}$ provides a first-order approximation of $1$, implying that $\bar{\omega}^{n+1}$ achieves only first-order accuracy. We therefore refer to \eqref{eqn:mrGSAV_ETDMS1} as the first-order exponential time-differencing mean-reverting SAV (\textbf{ETD-mr-SAV-MS1o}) scheme. The first-order companion does not use the quadratic ccSAV correction; it is introduced only as an embedded estimator sharing the same intermediate heat solves.

Due to its structural similarity to the ETD-mr-ccSAV-MS2o scheme, the intermediate values used to solve \eqref{eqn:mrGSAV_ETDMS2_u} can be reused to compute \textbf{ETD-mr-SAV-MS1o}. Specifically, the procedure is:
\begin{enumerate}
    \item Determine $\bar{r}^{n+1}$ by solving the linear relation
    \begin{equation} \label{eqn:step_5}
        \bar{r}^{n+1} = C^{n+1} - {{\tilde{\gamma}}}A^{n+1} + {{\tilde{\gamma}}}(1 - \bar{r}^{n+1})B^{n+1}.
    \end{equation}
    \item Obtain $\bar{\omega}^{n+1}$ from
    \begin{equation}\label{eqn:step_6}
         \bar{\omega}^{n+1} = \omega_1^{n+1} - (1 - \bar{r}^{n+1}) \omega_2^{n+1}.
    \end{equation}
\end{enumerate}

Combining the first- and second-order ETD-mr-SAV schemes above, we construct an embedded adaptive algorithm. Following \cite{hairer1993solving}, the adaptive strategy is summarized in {\bf Algorithm~\ref{alg:adaptive}} with the time-step update function
{
\begin{equation}\label{update_function}
A_{dp}(e_{\omega}, e_r, \tau)
= \rho \left(\min\left\{\left(\frac{\text{tol}_{\omega}}{e_{\omega}}\right)^{\frac{1}{2}},\ \frac{\text{tol}_{r}}{e_r}\right\}\right)\tau.
\end{equation}
}
Here, $\rho$ is a safety factor, and $\text{tol}_{\omega}$ and $\text{tol}_r$ are prescribed tolerances for the fluid variable $\omega$ and the scalar auxiliary variable $r$, respectively. These parameters depend on the specific problem under consideration. We refer to this adaptive time-stepping method as the \textbf{ETD-mr-ccSAV-MS12} scheme.

\begin{algorithm}[ht]
\footnotesize
\caption{ETD-mr-ccSAV-MS12 scheme}
\label{alg:adaptive}
\begin{algorithmic}
\State \textbf{Given:} $\omega^{n}, \omega^{n-1}$, $r^{n}$, $\tau_{n+1}$.
\State \textbf{Step 1.} Compute $\omega_1^{n+1}$, $\omega_2^{n+1}$ and the coefficients $A^{n+1}$, $B^{n+1}$, $C^{n+1}$ via \eqref{eqn:step_1} and \eqref{eqn:step_2}.
\State \textbf{Step 2.} Compute $r^{n+1}$ and $\bar{r}^{n+1}$ via \eqref{eqn:cubic} and \eqref{eqn:step_5}, respectively.
\State \textbf{Step 3.} Compute the second-order approximation $\omega^{n+1}$ and the first-order approximation $\bar{\omega}^{n+1}$ via \eqref{eqn:step_4} and \eqref{eqn:step_6}, respectively.
\State \textbf{Step 4.} Compute
$e_{\omega}^{n+1} = \frac{\Vert \bar{\omega}^{n+1} - \omega^{n+1} \Vert}{\max\{ \Vert \bar{\omega}^{n+1} \Vert,\ \Vert \omega^{n+1} \Vert\}}$
and $e_r^{n+1}=|r^{n+1}|$.
\If{$e_{\omega}^{n+1} \le  \text{tol}_{\omega}$ and $e_{r}^{n+1} \le  \text{tol}_r$}
    \State \textbf{Step 5.} Update
    $\tau_{n+2} \leftarrow \max\{\tau_{\min},\ \min\{A_{dp}(e_{\omega}^{n+1}, e_r^{n+1}, \tau_{n+1}),\ \tau_{\max}\}\}$.
\Else
    \State \textbf{Step 6.} Reset
    $\tau_{n+1} \leftarrow \max\{\tau_{\min},\ \min\{A_{dp}(e_{\omega}^{n+1}, e_r^{n+1}, \tau_{n+1}),\ \tau_{\max}\}\}$.
    \State \textbf{Step 7.} Go to \textbf{Step 1}.
\EndIf
\end{algorithmic}
\end{algorithm}

\begin{rem}
In this adaptive time-stepping setting, one may instead use a first-order single-step scheme obtained by replacing {{$\tilde{B}^{n+1/2}$ with $B(\psi^n, \omega^n)$}}. However, to construct an embedded adaptive algorithm without additional computational cost, we choose the first-order two-step ETD-mr-SAV scheme so that the intermediate computations can be shared with the ETD-mr-ccSAV-MS2o scheme.
\end{rem}
\subsection{Related schemes}
\subsubsection{ETD-ccSAV-MS2 and ETD-MS2}
There are two closely related schemes. The {\bf ETD-ccSAV-MS2 scheme} is a special case of the ETD-mr-ccSAV-MS2 scheme by setting the mean-reverting parameter $\gamma=0$ in \eqref{eqn:mrGSAV_ETDMS2_q}.
\begin{subequations}\label{eqn:ETD-SAV-MS2}
    \begin{align}
    &\omega^{n+1} = \varphi_0(\tau_{n+1} \nu \mathcal{L})\omega^{n}
    - \tau_{n+1}\big(1 - (r^{n+1})^2\big)\varphi_1(\tau_{n+1} \nu \mathcal{L})
    {\tilde{B}^{n+\frac{1}{2}}} \label{eqn:ETD-SAV-MS2_u}\\
    &\qquad\qquad\qquad\qquad\qquad
    + \tau_{n+1} \varphi_1(\tau_{n+1} \nu \mathcal{L})f^{n+\frac{1}{2}},\notag \\
    &r^{n+1} =  r^{n}
    + \tau_{n+1}(1 - r^{n+1}){{\tilde{\gamma}}}
    \Big\langle \varphi_1(\tau_{n+1} \nu \mathcal{L}){\tilde{B}^{n+\frac{1}{2}}} , \omega^{n+1} \Big\rangle.\label{eqn:ETD-SAV-MS2_q}
    \end{align}
\end{subequations}
This scheme will be utilized in our numerical experiments to isolate the effect of mean-reversion.
Moreover, we can eliminate SAV by setting $r^{n+1}=0$ in \eqref{eqn:mrGSAV_ETDMS2_u} and arrive at the following {\bf ETD-MS2 scheme} \cite{zhang1987schema}
\begin{equation}\label{ETD-MS2}
\omega^{n+1} = \varphi_0(\tau_{n+1} \nu \mathcal{L})\omega^{n}
    - \tau_{n+1}\varphi_1(\tau_{n+1} \nu \mathcal{L})
    {\tilde{B}^{n+\frac{1}{2}}}.
\end{equation}
This scheme will be utilized to test the effect of mean-reverting SAV.

\subsubsection{Alternative second-order schemes}
We also briefly describe two minor variants of the ETD-mr-ccSAV-MS2o scheme here.
\begin{enumerate}
\item We may remove the operator $\big(\varphi_1(\tau_{n+1}\gamma)\big)^{-1}\varphi_1(\tau_{n+1}\nu\mathcal{L})$ from \eqref{eqn:2nd_msSAV_approx} and arrive at the following alternative scheme to \eqref{eqn:mrGSAV_ETDMS2_u}--\eqref{eqn:mrGSAV_ETDMS2_q}:
\begin{equation*}
\begin{aligned}
&\omega^{n+1} = \varphi_0(\tau_{n+1} \nu \mathcal{L})\omega^{n}
- \tau_{n+1} (1 - (r^{n+1})^2) \varphi_1(\tau_{n+1} \nu \mathcal{L})
{\tilde{B}^{n+\frac{1}{2}}}\\
&\qquad\qquad\qquad\qquad\qquad+ \tau_{n+1} \varphi_1(\tau_{n+1} \nu \mathcal{L})f^{n+\frac{1}{2}},\\
&r^{n+1} = \varphi_0(\tau_{n+1} \gamma) r^{n}
+ \tau_{n+1}(1 - r^{n+1}) \varphi_1(\tau_{n+1} \gamma){{\tilde{\gamma}}}
\big\langle {\tilde{B}^{n+\frac{1}{2}}}, \omega^{n+1} \big\rangle.
\end{aligned}
\end{equation*}
This alternative scheme has the same second-order accuracy for $\omega^{n+1}$ and the same computational complexity. Moreover, the stability analysis is similar to that of Theorem~\ref{thm:etd_gsav_stability}, though it may require higher regularity of the solution.

\item Another choice is to approximate the nonlinear term at the 2nd-order extrapolated point $\tilde{\omega}^{n+\frac12}$, leading to the following second-order ETD-mr-ccSAV scheme:
\begin{equation*}
\begin{aligned}
&\omega^{n+1} = \varphi_0(\tau_{n+1} \nu \mathcal{L})\omega^{n}
- \tau_{n+1} (1 - (r^{n+1})^2) \varphi_1(\tau_{n+1} \nu \mathcal{L})
{B(\tilde{\psi}^{n+\frac{1}{2}}, \tilde{\omega}^{n+\frac{1}{2}})}\\
&\qquad\qquad\qquad\qquad\qquad
+ \tau_{n+1} \varphi_1(\tau_{n+1} \nu \mathcal{L})f^{n+\frac{1}{2}},\\
&r^{n+1} = \varphi_0(\tau_{n+1} \gamma) r^{n}
+ \tau_{n+1}(1 - r^{n+1}){{\tilde{\gamma}}}
\Big\langle \varphi_1(\tau_{n+1} \nu \mathcal{L})
{B(\tilde{\psi}^{n+\frac{1}{2}}, \tilde{\omega}^{n+\frac{1}{2}})}, \omega^{n+1} \Big\rangle.
\end{aligned}
\end{equation*}
This scheme exhibits the same stability results and similar numerical performance as \eqref{eqn:mrGSAV_ETDMS2_u}--\eqref{eqn:mrGSAV_ETDMS2_q}. Therefore, it is only briefly introduced here and will not be analyzed further.
\end{enumerate}

\begin{rem}[The velocity-pressure formulation]
The present analysis is restricted to the two-dimensional periodic vorticity formulation. In velocity-pressure variables, one may formally formulate analogous ETD-mr-ccSAV schemes using the Stokes operator and Leray projection; a detailed analysis for such formulations is beyond the scope of this work.
\end{rem}


\section{Long-time stability of the ETD-mr-ccSAV scheme}\label{sec:4}

The purpose of this section is to derive uniform-in-time energy bounds for the ETD-mr-ccSAV scheme proposed in the previous section. More precisely, we establish $L^\infty(0,\infty;L^2)$ stability (together with a corresponding cumulative dissipation estimate) under uniformly bounded forcing.

\medskip
We first recall several elementary properties of the two functions $\varphi_j$, $j=0,1$, that are related to ETD schemes.

\begin{lemma}\label{lem:phi_prop}
The functions $\varphi_i$ defined in \eqref{eqn:exp_relat} satisfy:
\begin{itemize}
    \item[(1)] $\varphi_i(z)$ is decreasing for $i=0,1$.
    \item[(2)] $z\varphi_1(z)+\varphi_0(z)=1$.
    \item[(3)] $0\le \varphi_0(z)\le 1$, $0\le \varphi_1(z)\le 1$, and
    $1+\frac{z}{2}\le (\varphi_1(z))^{-1}\le 1+z,\qquad \forall z\ge 0$.
    \item[(4)] $(\varphi_1(z))^{-1}\ge z$, \qquad $\forall z\in \mathbb{R}^+$.
\end{itemize}
\end{lemma}

\medskip
It is then straightforward to verify, with the help of the Fourier representation, that for any $h\in V$ and $\tau>0$,
\begin{equation}\label{eqn:phi_op_L}
\begin{aligned}
\varphi_1(\tau \nu \mathcal{L}) h
&=(\tau \nu \mathcal{L})^{-1}\bigl(I-e^{-\tau \nu \mathcal{L}}\bigr) h
=\sum_{\bm k\in \mathbb{Z}^2\setminus\{0\}}
\frac{1-e^{-\tau \nu \lambda_{\bm k}}}{\tau \nu \lambda_{\bm k}}\,
\hat{\bm h}_{\bm k}e^{\mathrm{i}\bm k\cdot \bm x},\\
(\varphi_1(\tau \nu \mathcal{L}))^{-1} h
&=\bigl(I-e^{-\tau \nu \mathcal{L}}\bigr)^{-1}(\tau \nu \mathcal{L}) h
=\sum_{\bm k\in \mathbb{Z}^2\setminus\{0\}}
\frac{\tau \nu \lambda_{\bm k}}{1-e^{-\tau \nu \lambda_{\bm k}}}\,
\hat{\bm h}_{\bm k}e^{\mathrm{i}\bm k\cdot \bm x},
\end{aligned}
\end{equation}
where $\lambda_{\bm k}$ denotes the eigenvalue of the operator $\mathcal{L}$ associated with the divergence-free Fourier mode
$\hat{\bm h}_{\bm k} e^{\mathrm{i}\bm k\cdot \bm x}$, with $\bm k\cdot \hat{\bm h}_{\bm k}=0$.

We also recall the Poincar\'e inequality
\begin{equation}\label{eqn:poincare}
\Vert  h \Vert \le \frac{1}{\sqrt{\lambda_1}} \Vert \mathcal{L}^{1/2} h \Vert,\qquad \forall  h\in V,
\end{equation}
where $\lambda_1>0$ is the smallest eigenvalue of $\mathcal{L}$.

\medskip
Combining the spectral representation \eqref{eqn:phi_op_L} with Lemma~\ref{lem:phi_prop}, we obtain the following simple estimates.

\begin{lemma}\label{lem:inv_opest}
For any $ h\in V$ and $\tau>0$, the following estimates hold:
\begin{eqnarray}
 & \Vert  h \Vert^2
= \tau \nu \Vert (\varphi_1(\tau \nu \mathcal{L}))^{1/2}\mathcal{L}^{1/2} h \Vert^2
+ \Vert (\varphi_0(\tau \nu \mathcal{L}))^{1/2} h \Vert^2, \label{eqn:op_est_1}
\\
\label{eqn:op_est_2}
& \Vert (\varphi_1(\tau \nu \mathcal{L}))^{1/2} h \Vert^2 \le \Vert  h \Vert^2,
\\
\label{eqn:inv_op_est_1}
& \Vert  h \Vert^2 + \frac{\tau \nu}{2}\Vert \mathcal{L}^{1/2} h \Vert^2
\le \Vert (\varphi_1(\tau \nu \mathcal{L}))^{-1/2} h \Vert^2
\le \Vert  h \Vert^2 + \tau \nu \Vert \mathcal{L}^{1/2} h \Vert^2,
\\
\label{eqn:inv_op_est_2}
& \Vert (\varphi_1(\tau \nu \mathcal{L}))^{-1/2} h \Vert^2
\ge \tau \nu \Vert \mathcal{L}^{1/2} h \Vert^2.
\end{eqnarray}
\end{lemma}

\subsection{Long-time stability for the ETD-mr-ccSAV-MS2o scheme}
We derive an $L^\infty(0,\infty;L^2)$ bound for the \textbf{ETD-mr-ccSAV-MS2o} scheme \eqref{eqn:mrGSAV_ETDMS2}. The same argument is also applicable to the \textbf{ETD-mr-SAV-MS1o} scheme. Throughout the proof we use Lemmas~\ref{lem:phi_prop}--\ref{lem:inv_opest} and the Poincar\'e inequality.

The proof relies on the exact cancellation between the ccSAV factor $1-(r^{n+1})^2$ in the vorticity equation and the factor $1-r^{n+1}$ in the auxiliary equation after testing by $r^{n+1}
+1$. The mean-reverting term then provides damping for the scalar variable, while the ETD heat operator provides damping for the vorticity.
\begin{theorem}\label{thm:etd_gsav_stability}
Assume $(\omega^i,r^i)\in (\dot{ H}^\alpha_{per}(\Omega)\cap  H)\times\mathbb{R}$ for $i=0,1$ with $\alpha > 0$, and $f\in L^\infty(0,\infty; H)$.
Let
\begin{equation}\label{theta}
\theta := \min\{\nu\lambda_1,\gamma\}>0,
\end{equation}
where $\lambda_1$ is the first eigenvalue of $\mathcal{L}$.
Then for all $n\ge1$, the ETD-mr-ccSAV-MS2o scheme is uniform-in-time stable in the sense that
\begin{equation}\label{eqn:L2_estimate_gsav}
{{\tilde{\gamma}}}\Vert \omega^{n+1} \Vert^2 + |r^{n+1}+1|^2
\le e^{-\theta \sum_{i=1}^n\tau_{i+1}}\bigl({{\tilde{\gamma}}}\Vert \omega^1\Vert^2 + |r^1+1|^2\bigr)
+ \frac{1}{\theta}\Bigl(\frac{{{\tilde{\gamma}}}}{\nu\lambda_1}\Vert f\Vert_{L^\infty(0,\infty; H)}^2 + \gamma\Bigr).
\end{equation}
\end{theorem}

\begin{proof}
We first show the solvability of the scheme.

Notice that $\omega\in H^\alpha$ implies $\psi\in H^{2+\alpha}$ by elliptic regularity. Therefore,  $\nabla^\perp\psi\in H^{1+\alpha}\subset L^\infty$. Hence $B(\psi,\omega)=\nabla \cdot(\nabla^\perp\psi \omega) \in V'$. Consequently, $\tilde{B}^{1+\frac{1}{2}}\in V'$ if $\omega^0,\omega^1\in H^\alpha$. 
This further implies that the intermediate values in computing $\omega^2$ are regular, i.e., $\omega^2_j\in V, j=1,2$ according to \eqref{eqn:step_1} and elliptic regularity. This leads to the finiteness of $A^2, B^2$ in \eqref{eqn:step_2}, and the solvability of $r^2$ in \eqref{eqn:cubic}. Therefore, ${\omega}^{2}\in V$ follows naturally from the linear combination in \eqref{eqn:step_4}. 
Hence, we can continue marching in time to solve for $\omega^3$ with $\alpha_{new}=\min\{\alpha, 1\}$. Consequently, the scheme can be solved step-by-step with $\omega^n\in V, \forall n\ge 2$.

Next, we derive the uniform-in-time bound on the numerical solution.
For this purpose, we need to carefully utilize the exponential functions associated with the ETD schemes as the ETD scheme does not contain any numerical dissipation from the linear part. This part of the proof is decomposed into six steps in order to enhance readability.

Step 1. [Reformulation of the scheme] We first rewrite the ETD-mr-ccSAV-MS2o scheme \eqref{eqn:mrGSAV_ETDMS2} as a variant of the Euler scheme using Lemma~\eqref{lem:phi_prop} (2) ($\varphi_0(z)=1-z\varphi_1(z)$) 
\begin{equation}\label{eqn:gsav_u_stab}
\begin{aligned}
\omega^{n+1}-\omega^n + \tau_{n+1}\nu\mathcal L\varphi_1(\tau_{n+1}\nu\mathcal L)\omega^n
= \tau_{n+1}\varphi_1(\tau_{n+1}\nu\mathcal L)\Bigl(f^{n+\frac12}-(1-(r^{n+1})^2){\tilde{B}^{n+\frac{1}{2}}}\Bigr),
\end{aligned}
\end{equation}
and
\begin{equation}\label{eqn:gsav_r_stab}
\begin{aligned}
&r^{n+1}-r^n+\tau_{n+1}\gamma\varphi_1(\tau_{n+1}\gamma)(r^n+1) \\
=&\tau_{n+1}(1-r^{n+1}){{\tilde{\gamma}}}\Bigl\langle \varphi_1(\tau_{n+1}\nu\mathcal L){\tilde{B}^{n+\frac{1}{2}}},\omega^{n+1}\Bigr\rangle
+\tau_{n+1}\gamma\varphi_1(\tau_{n+1}\gamma).
\end{aligned}
\end{equation}

Step 2. [Basic energy estimate on the reformulated scheme] We perform a standard energy estimate by considering $\int_\Omega\eqref{eqn:gsav_u_stab}\times 2{\tilde{\gamma}}\omega^{n+1} +
\eqref{eqn:gsav_r_stab}\times 2(r^{n+1}+1)$. Utilizing the identities $2(a-b,a)=\|a\|^2-\|b\|^2+\|a-b\|^2$ and
$2(a,b)=\|a\|^2+\|b\|^2-\|a-b\|^2$, and noticing that the nonlinear terms involving ${\tilde{B}^{n+\frac{1}{2}}}$ cancel since
\[
-(1-(r^{n+1})^2)+(1-r^{n+1})(r^{n+1}+1)=0,
\]
 we obtain
\begin{equation}\label{eqn:total_eqn}
    \begin{aligned}
        &{{\tilde{\gamma}}}\|\omega^{n+1}\|^{2} - {{\tilde{\gamma}}}\|\omega^{n}\|^{2} + {{\tilde{\gamma}}}\|\omega^{n+1} - \omega^{n}\|^{2}
        + {{\tilde{\gamma}}}\bigl\|(\tau_{n+1}\nu\mathcal L\varphi_1(\tau_{n+1}\nu\mathcal L))^{1/2}\omega^{n+1}\bigr\|^{2}
          \\&\quad
        +{{\tilde{\gamma}}}\bigl\|(\tau_{n+1}\nu\mathcal L\varphi_1(\tau_{n+1}\nu\mathcal L))^{1/2}\omega^{n}\bigr\|^{2}
        - {{\tilde{\gamma}}}\bigl\|(\tau_{n+1}\nu\mathcal L\varphi_1(\tau_{n+1}\nu\mathcal L))^{1/2}(\omega^{n+1}-\omega^{n})\bigr\|^{2} \\
        &\quad
        + |r^{n+1}+1|^{2} - |r^{n}+1|^{2} + |r^{n+1}-r^{n}|^{2}
        + \tau_{n+1}\gamma\varphi_1(\tau_{n+1}\gamma)|r^{n+1}+1|^{2}\\
        &\quad
        + \tau_{n+1}\gamma\varphi_1(\tau_{n+1}\gamma)|r^{n}+1|^{2}
        - \tau_{n+1}\gamma\varphi_1(\tau_{n+1}\gamma)|r^{n+1}-r^{n}|^{2} \\
        =&\; \Bigl\langle\frac{\mathcal{L}^{-1}}{\nu}\tau_{n+1}\nu\mathcal L\varphi_1(\tau_{n+1}\nu\mathcal L)f^{n+\frac12},\;2{{\tilde{\gamma}}}\omega^{n+1}\Bigr\rangle
        + \bigl\langle 2\tau_{n+1}\gamma\varphi_1(\tau_{n+1}\gamma),\;r^{n+1}+1\bigr\rangle.
    \end{aligned}
\end{equation}

Step 3. [Control of the destabilizing terms] We control the destabilizing terms, i.e., the sixth and twelfth terms on the left-hand side of \eqref{eqn:total_eqn} by the third and ninth terms on the left-hand side by utilizing Lemma~\eqref{lem:phi_prop} (2) and (3) ($\varphi_0(z)=1-z\varphi_1(z)$). More specifically, we have
\begin{align}
    &\|\omega^{n+1}-\omega^{n}\|^{2} 
    - \bigl\|(\tau_{n+1}\nu\mathcal{L}\varphi_{1}(\tau_{n+1}\nu\mathcal{L}))^{1/2} (\omega^{n+1}-\omega^{n})\bigr\|^{2} \nonumber \\
    &\quad = \|\varphi_0^{1/2}(\tau_{n+1}\nu\mathcal{L})(\omega^{n+1}-\omega^{n})\|^2 \ge 0, \label{eqn:omega_diff_nonneg} \\
    &|r^{n+1}-r^{n}|^{2} 
    - \tau_{n+1}\gamma\varphi_{1}(\tau_{n+1}\gamma)|r^{n+1}-r^{n}|^{2} \ge 0. \label{eqn:r_diff_nonneg}
\end{align}

Step 4. [Estimate of the right-hand side] We now estimate the right-hand side. For the first term, utilizing Cauchy--Schwarz, Young's and Poincar\'e inequalities, together with  the monotonicity of $\varphi_1$ and the fact that $\mathcal{L}\ge \lambda_1, \theta\le\nu\lambda_1$, we have
\begin{equation}\label{eqn:add_minus1}
\begin{aligned}
&\Bigl\langle\frac{\mathcal{L}^{-1}}{\nu}\,\tau_{n+1}\nu\mathcal{L}
             \varphi_{1}(\tau_{n+1}\nu\mathcal{L})f^{n+\frac12},\; 2{{\tilde{\gamma}}}\omega^{n+1}\Bigr\rangle \\
&\le {{\tilde{\gamma}}}\bigl\|(\tau_{n+1}\nu\mathcal{L}\varphi_{1}(\tau_{n+1}\nu\mathcal{L}))^{1/2}\omega^{n+1}\bigr\|^{2}
     + \frac{{{\tilde{\gamma}}}}{\nu^{2}}
       \bigl\|(\tau_{n+1}\nu\mathcal{L}\varphi_{1}(\tau_{n+1}\nu\mathcal{L}))^{1/2}
             \mathcal{L}^{-1}f^{n+\frac12}\bigr\|^{2} \\
&\le {{\tilde{\gamma}}}\bigl\|(\tau_{n+1}\nu\mathcal{L}\varphi_{1}(\tau_{n+1}\nu\mathcal{L}))^{1/2}\omega^{n+1}\bigr\|^{2}
     + {{\tilde{\gamma}}}\frac{\tau_{n+1}}{\nu}\varphi_{1}(\tau_{n+1}\nu\lambda_1)\|\mathcal{L}^{-\frac12}f^{n+\frac12}\|^{2} \\
&\le {{\tilde{\gamma}}}\bigl\|(\tau_{n+1}\nu\mathcal{L}\varphi_{1}(\tau_{n+1}\nu\mathcal{L}))^{1/2}\omega^{n+1}\bigr\|^{2}
     + {{\tilde{\gamma}}}\frac{\tau_{n+1}}{\nu\lambda_1}\varphi_{1}(\tau_{n+1}\theta)\|f^{n+\frac12}\|^{2} \\
&= {{\tilde{\gamma}}}\bigl\|(\tau_{n+1}\nu\mathcal{L}\varphi_{1}(\tau_{n+1}\nu\mathcal{L}))^{1/2}\omega^{n+1}\bigr\|^{2}
     + {{\tilde{\gamma}}}\frac{1-e^{-\theta\tau_{n+1}}}{\theta\nu\lambda_1}\|f^{n+\frac12}\|^{2}.
\end{aligned}
\end{equation}
Similarly, we deduce the following bound on the second source term on the RHS,
\begin{equation}\label{eqn:add_minus2}
\begin{aligned}
\bigl\langle 2\tau_{n+1}\gamma\varphi_{1}(\tau_{n+1}\gamma),\; r^{n+1}+1\bigr\rangle
&\le \tau_{n+1}\gamma\varphi_{1}(\tau_{n+1}\gamma)|r^{n+1}+1|^{2}
    + \tau_{n+1}\gamma\varphi_{1}(\tau_{n+1}\gamma) \\
&\le \tau_{n+1}\gamma\varphi_{1}(\tau_{n+1}\gamma)|r^{n+1}+1|^{2}
    + \gamma\frac{1-e^{-\theta\tau_{n+1}}}{\theta}.
\end{aligned}
\end{equation}

Step 5. [Simplified one-step energy inequality] Inserting  \eqref{eqn:add_minus1}-\eqref{eqn:add_minus2} into \eqref{eqn:total_eqn}, utilizing the cancellation of the two terms $\tau_{n+1}\gamma\varphi_{1}(\tau_{n+1}\gamma)|r^{n+1}+1|^{2}$ along with
${{\tilde{\gamma}}}\bigl\|(\tau_{n+1}\nu\mathcal{L}\varphi_{1}(\tau_{n+1}\nu\mathcal{L}))^{1/2}\omega^{n+1}\bigr\|^{2}$
since they appear on both sides of the inequality, and discarding the non-negative difference terms on the left-hand side of \eqref{eqn:total_eqn} according to \eqref{eqn:omega_diff_nonneg} and \eqref{eqn:r_diff_nonneg}, we deduce
\begin{multline}\label{eqn:gsav_energy_step}
{{\tilde{\gamma}}}\|\omega^{n+1}\|^2+|r^{n+1}+1|^2
\le {{\tilde{\gamma}}}\|\omega^n\|^2+|r^n+1|^2
-{{\tilde{\gamma}}}\tau_{n+1}\nu\Bigl\|(\mathcal L\varphi_1(\tau_{n+1}\nu\mathcal L))^{\frac12}\omega^{n}\Bigr\|^2 \\
-\tau_{n+1}\gamma\varphi_1(\tau_{n+1}\gamma)|r^n+1|^2
+ \frac{1-e^{-\theta\tau_{n+1}}}{\theta}\Bigl(\frac{{{\tilde{\gamma}}}}{\nu\lambda_1}\|f^{n+\frac12}\|^2+\gamma\Bigr).
\end{multline}
Recall that $1-z\varphi_1(z)=\varphi_0(z)$, we have, since $\theta=\min\{\nu\lambda_1,\gamma\}$,
\begin{equation}\label{eqn:u_decay_insert}
\begin{aligned}
&{{\tilde{\gamma}}}\|\omega^n\|^2-{{\tilde{\gamma}}}\tau_{n+1}\nu\Bigl\|(\mathcal L\varphi_1(\tau_{n+1}\nu\mathcal L))^{\frac12}\omega^{n}\Bigr\|^2\\
=&{{\tilde{\gamma}}}\bigl\|(\varphi_0(\tau_{n+1}\nu\mathcal L))^{\frac12}\omega^n\bigr\|^2
={{\tilde{\gamma}}}\|e^{-\tau_{n+1}\nu\mathcal L/2}\omega^n\|^2
\le {{\tilde{\gamma}}}e^{-\tau_{n+1}\nu\lambda_1}\|\omega^n\|^2
\le {{\tilde{\gamma}}}e^{-\tau_{n+1}\theta}\|\omega^n\|^2.
\end{aligned}
\end{equation}
Likewise,
\begin{equation}\label{eqn:r_decay_insert}
\begin{aligned}
&|r^n+1|^2-\tau_{n+1}\gamma\,\varphi_1(\tau_{n+1}\gamma)\,|r^n+1|^2\\
=&\varphi_0(\tau_{n+1}\gamma)\,|r^n+1|^2 =e^{-\tau_{n+1}\gamma}|r^n+1|^2 \le e^{-\tau_{n+1}\theta}|r^n+1|^2.
\end{aligned}
\end{equation}

Combining \eqref{eqn:gsav_energy_step}--\eqref{eqn:r_decay_insert}, we obtain the following one-step inequality
\begin{equation}\label{eqn:gsav_contract}
{{\tilde{\gamma}}}\|\omega^{n+1}\|^2+|r^{n+1}+1|^2
\le e^{-\theta\tau_{n+1}}\bigl({{\tilde{\gamma}}}\|\omega^n\|^2+|r^n+1|^2\bigr)
+\frac{1-e^{-\theta\tau_{n+1}}}{\theta}\Bigl(\frac{{{\tilde{\gamma}}}}{\nu\lambda_1}\|f^{n+\frac12}\|^2+\gamma\Bigr).
\end{equation}

Step 6. [The final estimate via induction]
Since $\|f^{k+\frac12}\|^2\le \|f\|_{L^\infty(0,\infty;H)}^2$, with the help of  \eqref{eqn:gsav_contract}, one can  verify via mathematical induction on $n$ the following inequality:
\begin{equation}\label{eqn:iteration_sum}
\begin{aligned}
{{\tilde{\gamma}}}\|\omega^{n+1}\|^2+|r^{n+1}+1|^2
&\le\; e^{-\theta\sum_{i=1}^{n}\tau_{i+1}}({{\tilde{\gamma}}}\|\omega^{1}\|^2+|r^{1}+1|^2) \\
& + \frac{1}{\theta}\Bigl(\frac{{{\tilde{\gamma}}}}{\nu\lambda_1}\|f\|_{L^\infty(0,\infty; H)}^2+\gamma\Bigr)\bigl(1-e^{-\theta\sum_{i=1}^{n}\tau_{i+1}} \bigr).
\end{aligned}
\end{equation}
We derive \eqref{eqn:L2_estimate_gsav} once we ignore $e^{-\theta\sum_{i=1}^{n}\tau_{i+1}} $ in the last term. This completes the proof.
\end{proof}

\begin{rem}
The case with $\alpha=0$ can be handled by invoking elliptic regularity with a source term in $W^{-1,p}, p\in (1, 2)$.
\end{rem}

\section{Numerical experiments}\label{sec:5}

In this section, we report the results of several numerical experiments that illustrate (1) the temporal accuracy, (2) the stability advantage over classical schemes, (3) the long-time stability, and (4) the practical performance of the proposed adaptive scheme, including its automatic time-step control capability.


All simulations are performed on a periodic box $\Omega=(0,L)\times(0,L)$. 
A Fourier spectral method with 256 modes in each direction is used for spatial discretization. Since the solutions to the NSE are analytic in space for positive time under spatially analytic forcing \cite{FoiasTemam1989}, the spatial discretization error is negligible when compared to the temporal discretization error due to spectral accuracy.

To characterize each flow regime, we use the Reynolds number defined as
$    \mathrm{Re} := \frac{LU}{\nu} = \frac{\|\bm{u}\|}{\nu}.$
Here,  $L$ as the characteristic length, and  $U=\|\bm{u}\|/L$. To monitor the flow dynamics and assess long-time stability, we report the enstrophy
\begin{equation}\label{eqn:enstrophy_def}
    \mathcal{E}(t) := \frac{\|\omega(\cdot,t)\|^2}{2}.
\end{equation}

\subsection{Second-order accuracy under fixed and variable time steps}
Here we confirm the second-order accuracy of our second-order ETD-mr-ccSAV-MS2 scheme under fixed and variable step sizes. We also compare its performance with that of the ETD-MS2 scheme \eqref{ETD-MS2} \cite{zhang1987schema}. 

{%
\begin{exm}[Accuracy test]\label{exm:kol_conv_test}
In this example, 
 we set $\nu = 1/50$, $f(x,y) = \cos(x)$, and we set the final time to $T = 1$. In addition,  $\gamma = 1000$ and $\tilde{\gamma} = 0.1$.
Inspired by the Kraichnan direct enstrophy cascade spectrum, the initial streamfunction is set to 
\begin{equation}\label{eqn:iso_pertbation}
\psi_{\varepsilon}(x,y) = \varepsilon \sum_{\substack{\bm{k}=(k_1,k_2)\in\mathbb{Z}^2 \\ 0<|\bm{k}|\leq 10}} \frac{1}{|\bm{k}|^3} \bigl[\cos(k_1 x)+\sin(k_1 x)\bigr]\bigl[\cos(k_2 y)+\sin(k_2 y)\bigr],
\end{equation}
with $\varepsilon = 2.5$. 
The initial auxiliary variable is set to its theoretical value, i.e., $r^0 = 0$. 
The initial Reynolds number 
is approximately 1198, rendering this test a moderate-to-large Reynolds number test. 
A reference solution is generated by the well-established ETDRK4 scheme \cite{kassam2005fourth} with a uniform time step of $\tau = 0.01\times 2^{-8}$.
Both ETD-mr-ccSAV-MS2 and ETD-MS2 are initialized with ETDRK4 in order to generate $\omega^1$.
\end{exm}

We first consider the fixed-step case. More specifically, we set $\tau = 0.01\times 2^{-k}$ with $k=2,3,\dots,8$, together with five uniformly spaced intermediate values between $k=2$ and $k=3$ in order to better resolve the \texttt{pre-asymp}totic regime, and to investigate the numerical blow-up of the classical ETD-MS2 scheme.

\begin{table}[htbp]
\label{tab:convergence_comparison}
\centering
\scriptsize
\caption{Velocity $L^2$ errors and convergence rates for $\nu=1/50$ ($Re_{in}\approx 1198$). 
}
\resizebox{\textwidth}{!}{
\begin{tabular}{l cc cc cc cc}
\toprule
& \multicolumn{4}{c}{$T=1.0$} & \multicolumn{4}{c}{$T=0.1$}\\
\cmidrule(lr){2-5}\cmidrule(lr){6-9}
& \multicolumn{2}{c}{ETD-MS2} & \multicolumn{2}{c}{ETD-mr-ccSAV-MS2o} & \multicolumn{2}{c}{ETD-MS2} & \multicolumn{2}{c}{ETD-mr-ccSAV-MS2o}\\
\cmidrule(lr){2-3}\cmidrule(lr){4-5}\cmidrule(lr){6-7}\cmidrule(lr){8-9}
Step size $\tau$ & Error & Rate & Error & Rate & Error & Rate & Error & Rate\\
\midrule
{5.00e-3}  & \texttt{NaN} & -- & {4.82e-2} & --    & {1.69e5}  & --    & {3.50e-3} & -- \\
{4.50e-3}  & \texttt{NaN} & -- & {4.33e-2} & \texttt{pre-asymp}  & {6.13e3}  & \texttt{pre-asymp} & {3.14e-3} & \texttt{pre-asymp} \\
{4.00e-3}  & \texttt{NaN} & -- & {3.68e-2} & \texttt{pre-asymp}  & {2.58}    & \texttt{pre-asymp} & {2.60e-3} & \texttt{pre-asymp} \\
{3.50e-3}  & \texttt{NaN} & -- & {2.85e-2} & \texttt{pre-asymp}  & {1.13e-1} & \texttt{pre-asymp} & {2.05e-3} & \texttt{pre-asymp} \\
{3.00e-3}  & \texttt{NaN} & -- & {1.70e-2} & \texttt{pre-asymp}  & {5.42e-3} & \texttt{pre-asymp} & {1.13e-3} & \texttt{pre-asymp} \\
{2.60e-3}* & \texttt{NaN} & -- & {5.75e-3} & \texttt{pre-asymp}  & {2.18e-4} & \texttt{pre-asymp} & {2.63e-4} & \texttt{pre-asymp} \\
{2.50e-3}  & {1.91e57} & --     & {2.35e-3} & \texttt{pre-asymp} & {6.61e-5} & \texttt{pre-asymp} & {6.75e-5} & \texttt{pre-asymp} \\
{1.25e-3}  & {1.08e-6} & \texttt{pre-asymp} & {1.00e-6} & \texttt{pre-asymp} & {1.10e-6} & \texttt{pre-asymp}  & {1.10e-6} & \texttt{pre-asymp} \\
{6.25e-4}  & {2.70e-7} & 2.00   & {2.57e-7} & 1.97  & {2.84e-7} & 1.95  & {2.83e-7} & 1.95 \\
{3.13e-4}  & {6.75e-8} & 2.00   & {6.48e-8} & 1.99  & {7.15e-8} & 1.99  & {7.13e-8} & 1.99 \\
{1.56e-4}  & {1.69e-8} & 2.00   & {1.63e-8} & 1.99  & {1.79e-8} & 2.00  & {1.79e-8} & 2.00 \\
{7.81e-5}  & {4.22e-9} & 2.00   & {4.07e-9} & 2.00  & {4.48e-9} & 2.00  & {4.47e-9} & 2.00 \\
\bottomrule
\end{tabular}}
\end{table}

Table \ref{tab:convergence_comparison} reports the velocity $L^2$ errors and convergence rates for the ETD-MS2 and ETD-mr-ccSAV-MS2o schemes at $T = 0.1$ and $T=1.0$. 
 We observe from this table the following phenomena. 
(i) 
For the small step size regime ($\tau \le 1.25 \times 10^{-3}$),  
we recover the expected second-order accuracy for both schemes. 
(ii) For intermediate step sizes ($5\times 10^{-3} \ge \tau\ge 2.6\times 10^{-3}$), the ETD-MS2 \eqref{ETD-MS2} without the stabilizing scalar auxiliary variable (SAV) can numerically blow up.
(iii) The ETD-mr-ccSAV-MS2 achieves an error of the order of $10^{-2}$ or less for the same intermediate step sizes. (iv) The rate of convergence  of the ETD-mr-ccSAV-MS2 scheme for the intermediate step sizes is not the theoretical one. 

We also tested the schemes at $T = 0.5$ with step size $\tau = 2.6 \times 10^{-3}$. The  ETD-MS2 scheme remains stable at $T = 0.5$ with an error of the order of $10^{-2}$, but suffers from numerical blow-up when the integration time is extended to $T = 1.0$. Due to space limitations, the results for $T = 0.5$ are omitted here. In contrast, the ETD-mr-ccSAV-MS2o scheme still produces a relatively small error, albeit without the theoretical order of convergence. This clearly demonstrates the second-order stability of the ETD-mr-ccSAV-MS2 scheme as well as its stability advantage in the regime where the classical ETD-MS2 scheme numerically blows up.

Next, we examine the variable-step case. Starting from a uniform grid, we perturb each uniform partition by a relative amplitude of $15\%$ and then rescale the resulting sequence so that the final time remains $T=1$. The errors are reported as functions of the total number of time steps $N$. As shown in Table~\ref{tab:convergency_perturbed}, the ETD-mr-ccSAV-MS2o scheme retains its expected second-order convergence under the perturbed time-step sequence. These results indicate that the proposed mean-reverting ccSAV treatment improves robustness for intermediate time steps while preserving the designed convergence order at relatively small time steps in both fixed-step and variable-step computations.

\begin{table}[htbp]
\centering
\scriptsize
\caption{Vorticity $L^2$ errors and observed convergence rates for $\nu=1/50$ and $m=1$, computed by ETD-mr-ccSAV-MS2o at $T=1$ under a $15\%$ perturbed variable-step sequence. Here $N$ denotes the total number of time steps.}
\label{tab:convergency_perturbed}
\begin{tabular}{@{} l *{7}{c} @{}}
\toprule
$N$ & $2^2\times100$ & $2^3\times100$ & $2^4\times100$ & $2^5\times100$ & $2^6\times100$ & $2^7\times100$ & $2^8\times100$ \\
\midrule
Error & 1.72e-2 & 6.63e-6 & 1.67e-6 & 4.17e-7 & 1.04e-7 & 2.61e-8 & 6.52e-9 \\
Rate & -- & \texttt{pre-asymp} & 1.99 & 2.00 & 2.00 & 2.00 & 2.00 \\
\bottomrule
\end{tabular}
\end{table}

\subsection{Stability advantage at larger time steps}\label{subsec:shearflow}
Example 5.1 indicates that the ETD-mr-ccSAV-MS2 scheme can maintain a small error when the classical ETD-MS2 numerically blows up. Here, we provide further numerical evidence for the stability advantage of the ETD-mr-ccSAV-MS2 scheme over the classical ETD-MS2 scheme \eqref{ETD-MS2} and ETD-ccSAV-MS2 \eqref{eqn:ETD-SAV-MS2} (without mean-reversion). 

{%
\begin{exm}\label{exm:shear_layer}[Double shear layer]
In this example, we consider the double shear layer problem on $\Omega = (0,1)^2$ with viscosity $\nu = 5\times 10^{-5}$ and final time $T = 1.2$. The initial velocity field is prescribed as
\begin{equation*}
u_0(x,y) = \begin{cases} \tanh\!\bigl(\rho(y - \tfrac{1}{4})\bigr), & y \le \tfrac{1}{2}, \\[4pt] \tanh\!\bigl(\rho(\tfrac{3}{4} - y)\bigr), & y > \tfrac{1}{2}, \end{cases} \qquad v_0(x,y) = \delta \sin(2\pi x),
\end{equation*}
with $\rho = 100$ and $\delta = 0.05$, and the initial vorticity $\omega_0$ is obtained from the velocity field via spectral differentiation. We consider a freely decaying flow with $f\equiv 0$.
The initial Reynolds number $\mathrm{Re}_{in}$ is $ \approx 1.96\times 10^4$. 
\end{exm}

We set  $\gamma = 1000$ and $\tilde{\gamma} = 0.1$. To test stability under a relatively large time step, we compare ETD-mr-ccSAV-MS2o and ETD-MS2 with $\tau = 6\times 10^{-4}$. The ETDRK4 solution computed with the finer time step $\tau = 1\times 10^{-4}$ is used as the reference. 
The evolution of the $L^2$ relative errors and the enstrophy for the two schemes are plotted in Figure~\ref{fig:shear_error_enstrophy}.
We observe that the ETD-mr-ccSAV-MS2o error remains below $10^{-2}$ over the whole interval and its enstrophy curve closely follows the reference trend. By contrast, the ETD-MS2 suffers numerical blow-up at approximately $t=0.831$. This is another example demonstrating ETD-mr-ccSAV-MS2's ability to produce a numerical solution with small ($\le 10^{-2}$) relative error when the classical ETD-MS2 blows up using the same time step.

\ignore{
\begin{figure}[htbp]
\centering
\includegraphics[width=0.6\linewidth]{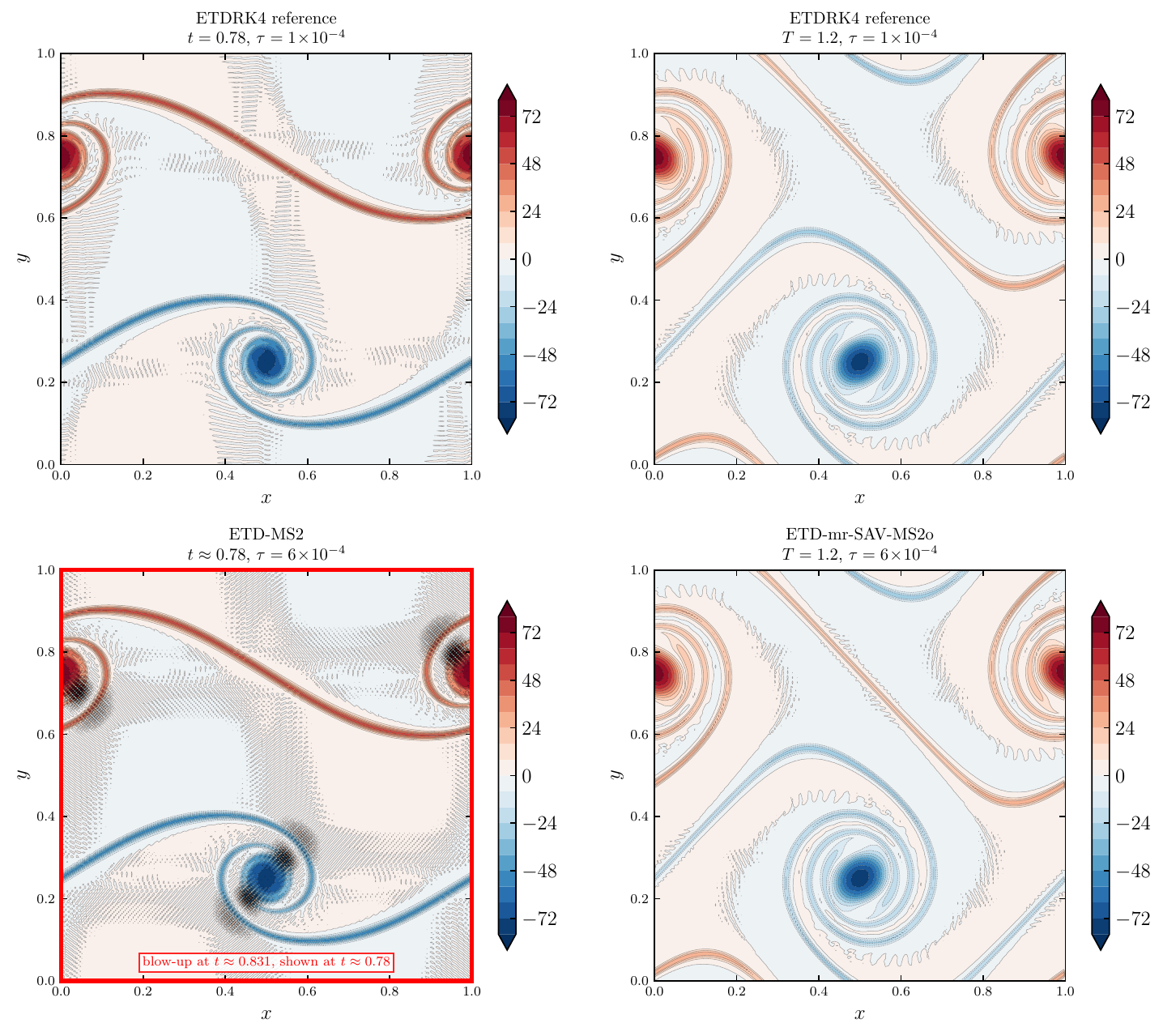}
\caption{Vorticity contours for the double shear layer with $\rho = 100$ and $\nu = 5\times 10^{-5}$. 
Top row: reference ETDRK4 solution computed with $\tau = 1\times 10^{-4}$.
Bottom row: ETD-MS2 before blow-up and ETD-mr-ccSAV-MS2o with $\tau = 6\times 10^{-4}$.
The ETD-MS2 panel is marked to indicate the loss of stability before the final time.}
\label{fig:shear_vorticity}
\end{figure}
}

\begin{figure}[htbp]
\centering
\includegraphics[width=0.8\linewidth]{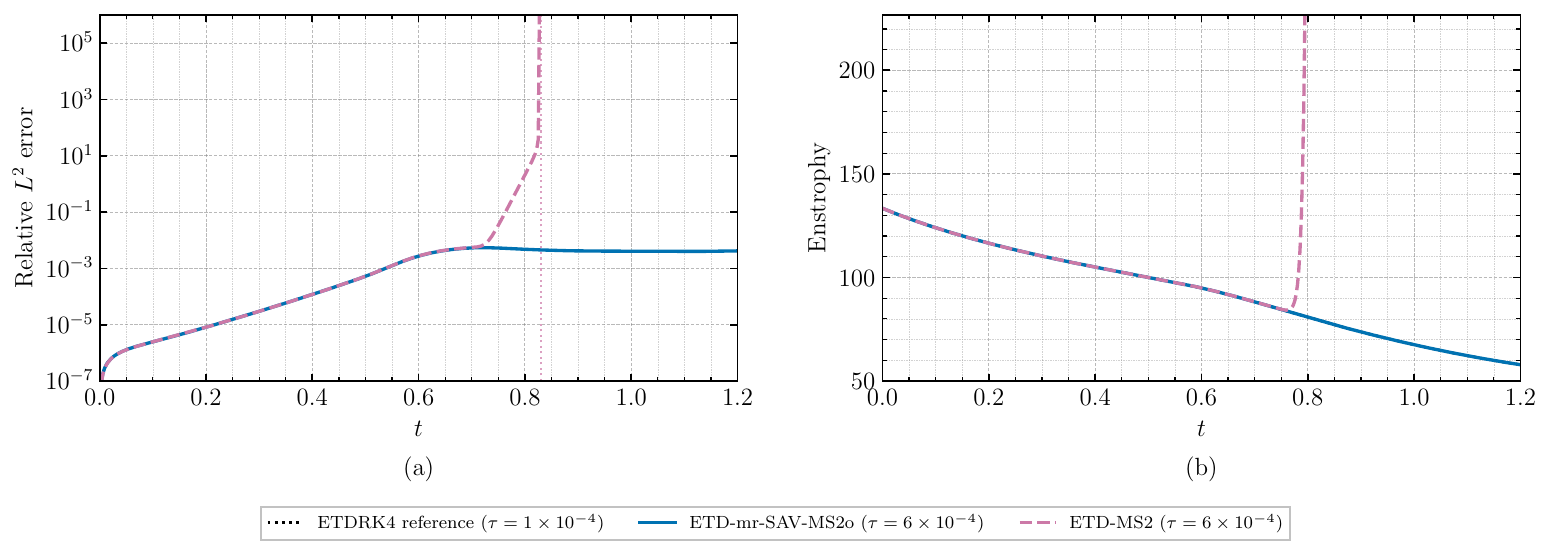}
\caption{Relative error and enstrophy evolution for the double shear layer with $\rho = 100$ and $\nu = 5\times 10^{-5}$. 
(a) Relative $L^2$ error with respect to the ETDRK4 reference solution computed with $\tau = 1\times 10^{-4}$. 
(b) Discrete enstrophy histories for the reference solution, ETD-mr-ccSAV-MS2o, and ETD-MS2. 
The ETD-mr-ccSAV-MS2o curve remains close to the reference up to $T = 1.2$, whereas ETD-MS2 develops rapidly growing error and loses stability near $t \approx 0.831$; non-finite states after blow-up are omitted from the plotted curves.}
\label{fig:shear_error_enstrophy}
\end{figure}
}

{%
We now demonstrate that the mean-reverting mechanism is not only theoretically important to our uniform-in-time bound, it can also enhance the performance by suppressing  numerical drifts. 
\begin{exm}\label{exm:mean_reverting_test}[Kolmogorov forcing: case 1]
We set $\Omega=(0,2\pi)^2$, $\nu=1/40$. The external forcing is given by $f=4\cos(4x)$ which is called a Kolmogorov forcing; the corresponding exact solution $10\cos(4x)$ is a Kolmogorov flow.  
To test the effect of the mean-reverting mechanism, we choose an initial condition that is different from the Kolmogorov flow. More specifically, we set  $\psi_0=\psi_\varepsilon$ with $\varepsilon=0.25$ in \eqref{eqn:iso_pertbation}, and we compare the same second-order ETD-mr-ccSAV discretization with $\gamma=1000$ and $\gamma=0$ under the time step $\tau=10^{-3}$  and $\tilde{\gamma}=0.1$. A reference solution is generated by running the ETDRK4 scheme with $\tau_{\rm ref}=1\times10^{-4}$.
\end{exm}

Figure~\ref{fig:mean_reverting_diagnostics} present the relative $L^2$ error of vorticity, and the evolution of enstrophy for the two choices of $\gamma$. Table~\ref{tab:mean_reverting_error_comparison} lists the corresponding values at selected times instants $t=4,6,8, 10$. 
Here are a few observations.
(i) The strong mean-reversion leads to a substantially more accurate numerical solution when compared to the no-mean-reversion case. With $\gamma=1000$, the relative $L^2$ error is of the order of $10^{-3}$, whereas the non-mean-reverting computation with $\gamma=0$ leads to a relative error $\ge 50\%$ at the end of the simulation although it is small at early times. 
(ii) The enstrophy or enstrophy production rate may not be strongly correlated with solution error. Indeed, the relative error in enstrophy for the no mean-reversion case of $\gamma=0$ at $t=6$ is  small at less than $0.15\%$. However, the relative $L^2$ error in solution is already more than $10\%$.
This result indicates that enstrophy-related metrics may serve as auxiliary stability diagnostics, yet they may not be suitable as standalone error indicators for adaptive algorithms.


\begin{figure}[htbp]
\centering
\includegraphics[width=0.8\linewidth]{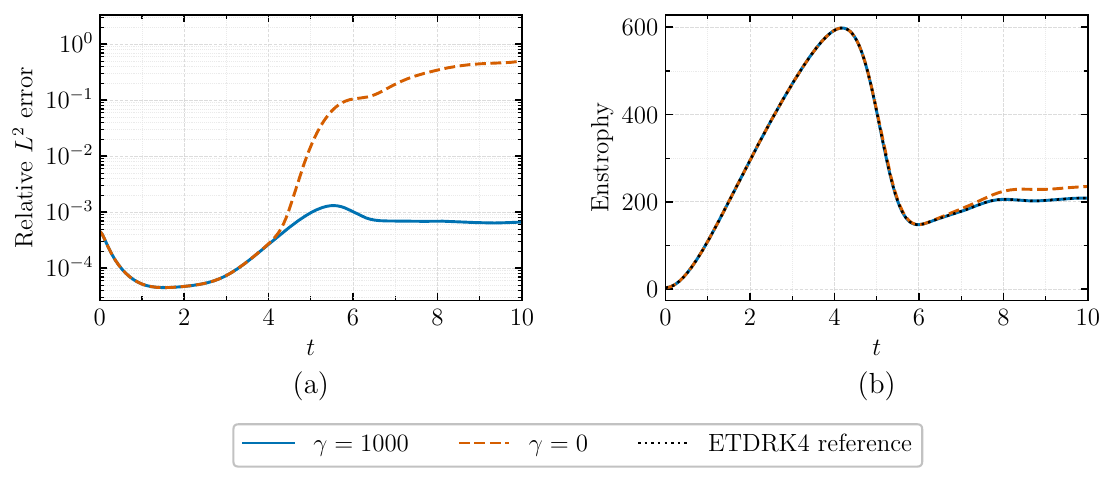}
\caption{Effect of the mean-reverting correction.
(a) Relative $L^2$ error of the vorticity. 
(b) Enstrophy evolution for the reference solution and the two ETD-mr-ccSAV variants. 
}
\label{fig:mean_reverting_diagnostics}
\end{figure}

\begin{table}[htbp]
\centering
\scriptsize
\setlength{\tabcolsep}{4pt}
\caption{Comparison of the relative errors in vorticity, and enstrophy at selected times for the mean-reverting and non-mean-reverting ccSAV schemes. All errors are computed with respect to the ETDRK4 reference solution. Here, $e_\omega=\|\omega_h-\omega_{\rm ref}\|_2/\|\omega_{\rm ref}\|_2$, $e_{\mathcal{E}}=|\mathcal{E}_h-\mathcal{E}_{\rm ref}|/|\mathcal{E}_{\rm ref}|$.
}
\label{tab:mean_reverting_error_comparison}
\begin{tabular}{@{}c cc cc@{}}
\toprule
\multirow{2}{*}{$t$}
& \multicolumn{2}{c}{$\gamma=1000$}
& \multicolumn{2}{c}{$\gamma=0$} \\
\cmidrule(lr){2-3}\cmidrule(lr){4-5}
& $e_{\omega}$ & $e_{\mathcal{E}}$
& $e_{\omega}$ & $e_{\mathcal{E}}$ \\
\midrule
$4$  & $2.68\times10^{-4}$ & $1.48\times10^{-7}$
     & $2.75\times10^{-4}$ & $2.64\times10^{-6}$ \\
$6$  & $1.02\times10^{-3}$ & $2.27\times10^{-7}$
     & $1.06\times10^{-1}$ & $1.48\times10^{-3}$ \\
$8$  & $6.95\times10^{-4}$ & $1.50\times10^{-5}$
     & $3.47\times10^{-1}$ & $8.92\times10^{-2}$ \\
$10$ & $6.68\times10^{-4}$ & $7.06\times10^{-5}$
     & $5.07\times10^{-1}$ & $1.28\times10^{-1}$ \\
\bottomrule
\end{tabular}
\end{table}
}

{
\subsection{Performance of the adaptive scheme}
Here we illustrate the performance of our adaptive ETD-mr-ccSAV-MS12 scheme.

\begin{exm}\label{exm:kolmogorov_1}[Kolmogorov forcing: case 2, adaptive time-stepping]
We test the performance of the adaptive algorithm in a Kolmogorov forcing setting. More specifically, we set $\Omega= (0,2\pi)^2$,  $f(x,y) = -m\cos(my), m=4, \nu=1/40$. 
The stream function of the associated steady-state Kolmogorov flow is given by 
$\psi(x,y)=-\frac{1}{\nu m3}\cos(my)$.
This Kolmogorov flow is unstable in the parameter regime that we test here \cite{armbruster1996symmetries}. 

In addition, we set $\gamma=1000$, and $\tilde{\gamma}=0.1$ in our scheme, and the final time is set to $T=20$.
The initial streamfunction is prescribed as $\psi_0(x,y) = \psi_\varepsilon(x,y)$ with $\varepsilon=3$, yielding an initial Reynolds number $\text{Re}_{in} \approx 1156.99$ rendering the simulation a moderate-Reynolds-number one. Two time-stepping strategies are compared: (1) the uniform step sizes $\tau = 0.004,0.002,0.0015,0.001, 0.0005$ are used for the ETD-mr-ccSAV-MS2o scheme; and (2) the adaptive ETD-mr-ccSAV-MS12 scheme with the parameters in the step-controller \eqref{update_function} set to be
\begin{equation}\label{adaptive_set}
    \rho = 0.9,\quad \mathrm{tol}_{\omega} = \mathrm{tol}_{r} = 5.5\times 10^{-4},\quad
    \tau_{\min} = 1\times 10^{-5},\quad \tau_{\max} = 1\times 10^{-2}.
\end{equation}
\end{exm}
 
Figure~\ref{fig:long_time_adap} reports the adaptive step-sizes, culmulative step-counts, CPU times, the auxiliary variable, local $L^2$ error indicator, and the relative reference $L^2$ errors. 
(i) Panel (a) indicates that the step-size controller changes the time step according to the local solution activity and the prescribed tolerances. 
(ii) 
Panel (f) shows that the adaptive scheme attains an error level comparable to the finest fixed-step implementation at $\tau=5\times10^{-4}$.
(iii) Panel (b) and (c) demonstrate that the adaptive scheme uses fewer cumulative steps and less CPU time than the fixed-step scheme with similar relative errors. 
(iv) Panel (d) and (e) show the adaptive scheme's ability to control the scalar auxiliary variable and the relative error indicator in the scheme.
We conclude that the adaptive method provides a more efficient balance between accuracy and computational effort for this test case.

{
We further test the robustness of the adaptive step-size controller by perturbing the nominal tolerance pair $(\mathrm{tol}_{\omega},\mathrm{tol}_r)=(5.5\times10^{-4},5.5\times10^{-4})$ by $\pm10\%$.
The changes in CPU time and relative errors are all within $15\%$. The test does not imply complete tolerance independence, but it demonstrates local robustness around the tolerance level used in this experiment.}


\begin{figure}[htbp]
\centering
\includegraphics[width=0.9\linewidth]{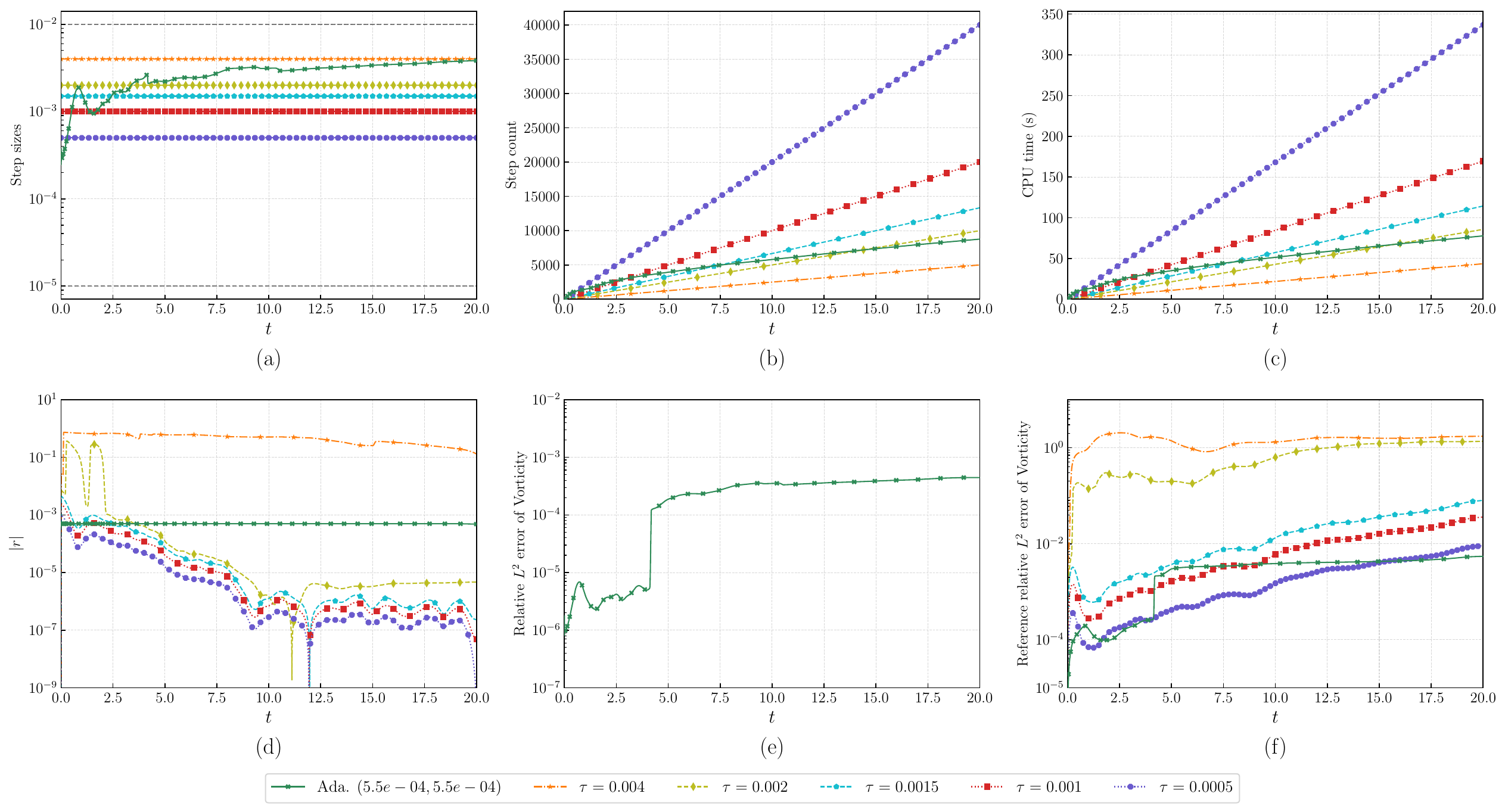}
\caption{Comparison of fixed-step and adaptive ETD-mr-ccSAV-MS12 schemes over $t \in [0, 20]$ for  Kolmogorov forcing with $\nu = 1/40$ and $m = 4$.
The adaptive solver uses $\mathrm{tol}_{\omega}=\mathrm{tol}_{r}=5.5\times10^{-4}$ and $\rho=0.9$; fixed-step results are shown for $\tau=5\times10^{-4},10^{-3}$, $1.5\times10^{-3}$, $2\times10^{-3}$, and $4\times10^{-3}$. 
(a) Adaptive time step sizes as a function of time; horizontal dashed lines mark the fixed step sizes and the solver bounds $\tau_{\min}=10^{-5}$ and $\tau_{\max}=10^{-2}$. (b) Cumulative number of steps. (c) CPU time. (d) Absolute value of the auxiliary variable $|r|$. (e) Relative $L^2$ error indicator of vorticity used by the adaptive controller. (f) Reference relative $L^2$ error of vorticity, computed against the ETDRK4 solution with $\tau_{\rm ref}=2.5\times10^{-4}$.}
\label{fig:long_time_adap}
\end{figure}

{

\subsection{Long-time efficiency and statistics}
The next  example examines whether the stability mechanism remains useful in adaptive and genuinely long-time computations, where trajectory-wise agreement is not expected but statistical diagnostics remain meaningful.
\begin{exm} [Kolmogorov forcing: case 3, long-time]
To examine the long-time behavior in a dynamically unstable regime, we use the same parameters as in the previous example \ref{exm:kolmogorov_1}, i.e., $m=4$, $\nu=1/40$, $\gamma=1000$,  $\tilde{\gamma} = 0.1$, $\psi_0=\psi_\varepsilon, \varepsilon=3$, but with a much larger terminal time $T=10000$. 
We compare a fixed-step ETD-mr-ccSAV-MS2o computation with $\tau=5\times 10^{-4}$ and an adaptive ETD-mr-ccSAV-MS12 computation with the controller parameters specified in \eqref{adaptive_set}. 
This choice of fixed-step size is suggested by Example 5.4 as this step size produces roughly equivalent small local relative error in vorticity when compared to the adaptive scheme with the chosen parameters.
\end{exm}

Figure~\ref{fig:adaptive_step_k} summarizes the long-time diagnostics. Panel~(a) illustrates the variation in step sizes of the adaptive computation. 
{
To quantify the response of the adaptive controller to changes in the solution, we compare the adaptive step size $\tau_n$ with the enstrophy production rate $\frac{d}{dt}\mathcal{E}$ along the adaptive trajectory. The Pearson correlation coefficient between these two quantities is $\text{PCC}=-0.6175$. This supports the intuition that an effective step-size controller should take small time-steps when the temporal evolution of the solution is fast.} Panel (b) compares the enstrophy histories produced by the two computations. Although the two trajectories separate over long time, both remain uniformly bounded in time as predicted by the scheme, even in this weakly turbulent regime \cite{armbruster1996symmetries}. Panel (c) plots the long-time distributions of the enstrophy associated with the two schemes. They are visually very close.  Panel (d) reports the cumulative CPU time as a function of physical time. The adaptive computation reaches $T=10000$ using approximately $1.49\times10^6$ time steps and $7.86$ hours of wall-clock time, whereas the fixed-step computation requires $2\times10^7$ time steps and approximately $125.08$ hours. Thus, the adaptive strategy reduces the computational cost by a factor of about $15.9$ while providing good  approximations of the long-time statistical quantities.

\begin{figure}[htbp]
    \centering
    \includegraphics[width=0.9\linewidth]{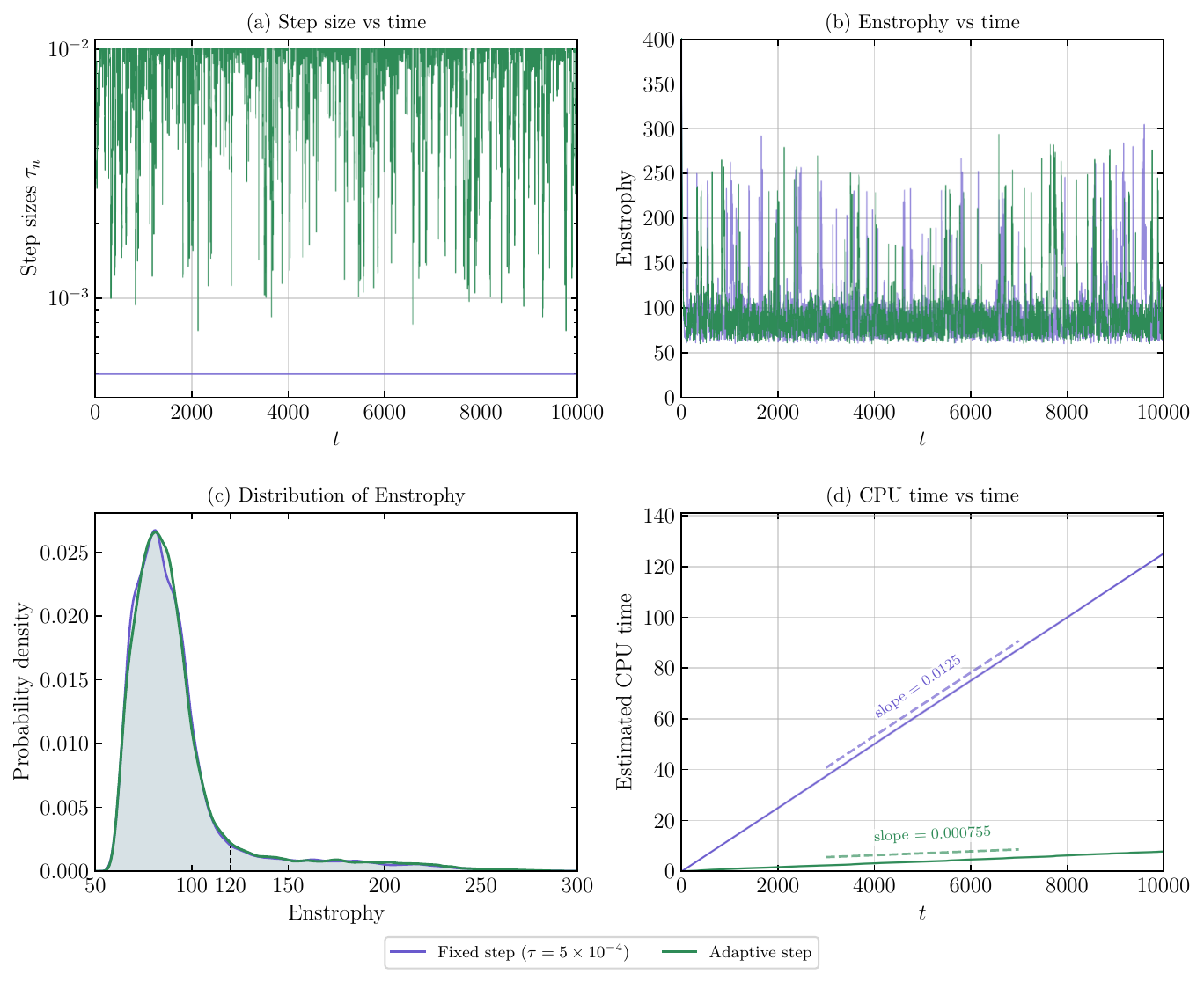}
    \caption{Long-time diagnostics in the bursting regime with $\nu=1/40$ and $m=4$. (a) Adaptive step sizes $\tau_n$ as a function of time, with the fixed step $\tau=10^{-3}$ shown for reference. (b) Enstrophy histories from the fixed-step and adaptive-step computations. (c) Probability density functions of enstrophy estimated on the statistically steady interval. (d) CPU time as a function of physical time for the fixed-step and adaptive-step computations.}
    \label{fig:adaptive_step_k}
\end{figure}

{

Notice that panel~(b) shows the existence of multiple quiescent windows and neighboring bursting events in the enstrophy evolution. We define quiescent windows as intervals during which the enstrophy remains below $\mathbb{E}({\mathcal{E}})+ \text{Std}(\mathcal{E})\approx125$; their mean duration is about $20$ eddy turnover times, so they are not short transients. In contrast, bursts are identified when the enstrophy exceeds $\mathbb{E}({\mathcal{E}})+2\text{Std}(\mathcal{E})\approx153$. The bursting events typically last  only $1 \sim 2$ eddy turnover times. Here, the eddy-turnover time is defined as $T_e=L/U$, with $U = \sqrt{\frac{2}{L^2T}\int_0^{T} E(t)\,dt}$, and $E(t)=\frac12\int_\Omega |\nabla^\perp\psi(\cdot,t)|^2 .$ 
In particular, an integration over $\mathcal{O}$(10) eddy-turnover times can fall entirely inside a quiescent window and therefore miss the subsequent transition to a bursting episode. This observation underscores that, for intermittent regimes, “long time” should be understood relative to the recurrence time of dynamically important events, not merely relative to a few eddy-turnover times.

}

We also present a comparison of the long-term statistical characteristics of enstrophy distributions in Table \ref{tab:moment_flow_scales_comparison}. The relative deviations between the mean value, second-order raw moment, variance, and standard deviation derived using the adaptive scheme and those from the fixed step size approach are small. This demonstrates that the adaptive step selection strategy induces no substantial bias in either the mean enstrophy or its fluctuation magnitude.

\begin{table}[htbp]
\centering
\scriptsize
\caption{Comparison of long-time statistics of the enstrophy and time-averaged flow scales for the fixed-step and adaptive-step computations with $\nu=1/40$ and $m=4$. Flow scales are averaged over the statistically steady interval $1000<t<10000$.}
\label{tab:moment_flow_scales_comparison}
\begin{tabular}{@{}lccc@{}}
\toprule
Quantity & Fixed step & Adaptive step & Relative difference \\
\midrule
\multicolumn{4}{@{}l}{\textit{long-time statistics of the enstrophy}} \\
Mean $\mathbb{E}[\mathcal{E}]$ & $92.4689$ & $93.2755$ & $0.872\%$ \\
Second moment $\mathbb{E}[\mathcal{E}^2]$ & $9490.7287$ & $9703.0214$ & $2.237\%$ \\
Variance $\mathrm{Var}(\mathcal{E})$ & $940.2255$ & $1002.7027$ & $6.645\%$ \\
Standard deviation $\mathrm{Std}(\mathcal{E})$ & $30.6631$ & $31.6655$ & $3.269\%$ \\
\midrule
\multicolumn{4}{@{}l}{\textit{Time-averaged flow scales}} \\
Typical velocity   & $1.1687$   & $1.1683$  & $0.034\%$ \\
Reynolds number    & $293.7226$ & $293.6217$ & $0.034\%$ \\
Eddy turnover time & $5.3763$   & $5.3781$  & $0.034\%$ \\
\bottomrule
\end{tabular}
\end{table}

The distributional discrepancy is also small. Table~\ref{tab:tv_comparison} reports the total variation distances, and relative $L^1$ errors for the low-enstrophy range $\mathcal{E}<120$, the bursting-tail range $\mathcal{E}\geq120$, and the full resolved range. These results show that the adaptive computation approximates the long-time enstrophy distribution of the fixed-step reference reasonably well except for the tail part. The relatively large relative error in the tail or rare events is expected due to well-known challenges in simulating tail events.

\begin{table}[htbp]
\centering
\scriptsize
\caption{Total variation distance between the steady-state enstrophy distributions obtained by ETD-mr-ccSAV-MS2o and ETD-mr-ccSAV-MS12 for $\nu=1/40$ and $m=4$.}
\label{tab:tv_comparison}
\begin{tabular}{@{}lcccc@{}}
\toprule
Enstrophy range &  Total variation distance & Relative $L^1$ error \\
\midrule
$\mathcal{E}<120$ &  $0.018273$ & $4.070\%$ \\
$\mathcal{E}\geq120$ & $0.005083$ & $9.968\%$ \\
$-\infty \leq \mathcal{E} \leq \infty$ &  $0.023355$ & $4.671\%$ \\
\bottomrule
\end{tabular}
\end{table}

Finally, we present a few statistics on tail events. Since we are interested in the long-time statistics, we skip the first 1000 units of spin-up time and focus on $t\geq1000$.
A tail event is defined as a rare event in the tail of the distribution in which the enstrophy exceeds the prescribed threshold $\mathbb{E}[\mathcal{E}]+k\operatorname{Std}(\mathcal{E})$. More precisely, each tail event corresponds to one connected time interval over which this threshold is exceeded, where the mean and standard deviation are those reported in Table~\ref{tab:moment_flow_scales_comparison}.
Table~\ref{tab:burst_count_thresholds} gives the counts for $k=1,2,3,4,5$, together with the corresponding thresholds and the absolute count error relative to the fixed-step computation. 
The maximum relative error is below $12\%$. Although these tail event counts should not be interpreted as high-precision estimates, the agreement is reasonable for rare-event diagnostics.

\begin{table}[htbp]
\centering
\scriptsize
\caption{Tail event counts for thresholds $\mathbb{E}[\mathcal{E}]+k\operatorname{Std}(\mathcal{E})$ after $t\geq1000$.}
\label{tab:burst_count_thresholds}
\begin{tabular}{@{}cccccc@{}}
\toprule
$k$ & Fixed-step tail event count & Adaptive-step tail event count & count error \\
\midrule
$1.0$  & $115$  & $119$ & $3.48\%$ \\
$2.0$  & $86$  & $88$ & $2.33\%$ \\
$3.0$  & $80$  & $81$ & $1.25\%$ \\
$4.0$  & $60$  & $67$ & $11.67\%$ \\
$5.0$  & $22$  & $21$ & $4.55\%$ \\
\bottomrule
\end{tabular}
\end{table}

The tail statistics are naturally more sensitive than the mean and variance, and the burst-count comparison should therefore be interpreted as a qualitative rare-event diagnostic rather than a high-precision estimate.

These numerical tests are not intended as a proof of convergence of invariant measures. Rather, they assess whether the adaptive long-time stable scheme reproduces commonly used statistical diagnostics of a reliable fixed-step computation at substantially reduced computational cost.
}

\section{Conclusion}\label{sec:6}

We have developed an efficient, unconditionally long-time stable, variable-step, second-order exponential time-differencing (ETD) scheme for the incompressible Navier–Stokes equations in periodic domains, together with two small variants. The proposed methods are long-time stable in the sense that the numerical solutions admit uniform-in-time bounds in the space $L^\infty(0,\infty;L^2)$ under general bounded external forcing, independent of the time-step size and the viscosity parameter.
This stability property is a necessary foundation for reliable long-time computation and for the numerical investigation of statistical quantities and asymptotic regimes.

Our approach integrates two recently developed generalizations of the scalar auxiliary variable approach, a mean-reverting scalar auxiliary variable (mr-SAV) formulation and a concurrent-correction SAV (ccSAV), together with multistep ETD techniques. This combination plays a central role in ensuring unconditional long-time stability while maintaining second-order temporal accuracy under variable step sizes. The resulting scheme treats the nonlinear advection term explicitly and remains solvable for arbitrary parameter values.

From a computational perspective, the proposed method is highly efficient, requiring only two heat solves and the solution of a single scalar cubic algebraic equation at each time step. An embedded adaptive time-stepping strategy, based on a first-order companion scheme, has also been developed to automatically adjust the time-step size in response to the evolving dynamics, thereby offering a way to improve efficiency without sacrificing stability.

The numerical experiments presented in this work confirm the theoretical analysis. In particular, they demonstrate the expected second-order temporal convergence, unconditional long-time stability under variable time stepping, enhanced stability when compared to the classical ETD-MS2 as well as the ETD-ccSAV-MS2 (without mean reversion), and the ability to capture long-time statistics in our Kolmogorov forcing test problem. Moreover, the adaptive strategy effectively balances accuracy and efficiency, leading to computational savings in our test case with moderate Reynolds number.

While the present results are encouraging, several important directions merit further investigation. A rigorous convergence theory and invariant-measure convergence analysis for the variable-step scheme remain important directions for future work.

The unconditional long-time stability and statistical consistency of the proposed ETD-mr-ccSAV framework also make it well suited for integration with modern data-driven and machine learning methodologies in fluid dynamics. 
In particular, the ability of the proposed method to efficiently resolve intermittent and extreme events could potentially be used for the systematic investigation of rare-event dynamics in turbulent and transitional flows. 

These directions suggest that structure-preserving ETD-mr-ccSAV methods may serve as a foundation for hybrid physics–machine learning frameworks aimed at long-time forecasting, statistical inference, and data-driven discovery in complex fluid systems.

\section*{Acknowledgments}
This work is supported in part by NSFC12271237. The second author acknowledges helpful conversations with Professors Wenbin Chen and Qiang Du. 

\section*{Declarations}
The authors declare that they have no known competing financial interests or personal relationships that could have appeared to influence the work reported in this paper.

}}

\bibliographystyle{elsarticle-num} 
\bibliography{references}

\end{document}

%% file: paper_shared.tex
\usepackage{bm}
\usepackage{color}
\usepackage{amsfonts}
\usepackage{amsmath}
\usepackage{amssymb}
\usepackage{graphicx}
\usepackage{subfigure}
\usepackage{listings} 
\usepackage{multirow}
\usepackage{enumitem}
\usepackage{indentfirst}
\usepackage{algorithm}
\usepackage{algpseudocode}
\usepackage{booktabs}
\usepackage{appendix}
\usepackage{comment}

\newsiamremark{rem}{Remark}
\newsiamremark{exm}{Example}

\numberwithin{equation}{section}

\newcommand{\ignore}[1]{}

\title{Unconditionally Long-Time Stable Variable-Step Second-Order Exponential Time-Differencing Schemes for the 2D Periodic Incompressible NSE\thanks{Submitted to the editors DATE.}}

\author{
  Haifeng Wang\thanks{School of Mathematical Sciences, Eastern Institute of Technology, Ningbo, Zhejiang 315200, China (\email{hfwang@eitech.edu.cn}).}
  \and
  Xiaoming Wang\thanks{School of Mathematical Sciences, Eastern Institute of Technology, Ningbo, Zhejiang 315200, China (\email{wxm@eitech.edu.cn}). Corresponding author.}
  \and
  Min Zhang\thanks{School of Mathematical Sciences, Shanghai Jiao Tong University, Shanghai 200240, China, and School of Mathematical Sciences, Eastern Institute of Technology, Ningbo, Zhejiang 315200, China (\email{zhangmmm@sjtu.edu.cn}).}
}